\documentclass{amsart}

\usepackage{amsfonts,amsmath,amsthm,amssymb,amscd,latexsym,enumerate,etoolbox,mathrsfs,comment}

\usepackage[all,cmtip]{xy}
\usepackage{thmtools, thm-restate}
\usepackage{tikz-cd}
\usepackage{tikz}
\usetikzlibrary{calc}

\usepackage[pagebackref,colorlinks,linkcolor=blue,citecolor=blue,urlcolor=blue,hypertexnames=true]{hyperref}

\newcommand{\C}{\mathbb{C}} 
\newcommand{\N}{\mathbb{N}}

\newcommand{\Z}{{\mathbb Z}}
\newcommand{\Q}{{\mathbb Q}}

\renewcommand{\phi}{\varphi}

\theoremstyle{plain}
    \newtheorem{theorem}{Theorem}[section]
    \newtheorem{lemma}[theorem]{Lemma}
    \newtheorem{corollary}[theorem]{Corollary}
    \newtheorem{proposition}[theorem]{Proposition}
        \newtheorem*{proposition*}{Proposition}
    \newtheorem{conjecture}[theorem]{Conjecture}
\theoremstyle{definition}
    \newtheorem{definition}[theorem]{Definition}
    \newtheorem{example}[theorem]{Example}

    \newtheorem{remark}[theorem]{Remark}
    \newtheorem{remarks}[theorem]{Remarks}
    \newtheorem{examples}[theorem]{Examples}
\theoremstyle{remark}

\newcommand{\dcc}{\mathcal{D}_{cch}^+}

\newcommand{\dch}{\mathcal{D}_{ch}^+}

\newcommand{\Manoa}{M\=anoa}
\newcommand{\Hawaii}{Hawai\kern.05em`\kern.05em\relax i}

\begin{document}

\title[The rational HK-conjecture]{The rational HK-conjecture: transformation groupoids and a revised version}
\author{Robin J. Deeley}
\address{Robin J. Deeley,   Department of Mathematics,
University of Colorado Boulder
Campus Box 395,
Boulder, CO 80309-0395, USA }
\email{robin.deeley@colorado.edu}
\author{Rufus Willett}
\address{Rufus Willett, Department of Mathematics,
University of Hawaii Manoa
2565 McCarthy Mall, Keller 401A
Honolulu, HI 96822, USA}
\email{rufus@math.hawaii.edu}
\thanks{RJD was partially supported by NSF Grants DMS 2000057 and 2247424 and Simons Foundation Gift MP-TSM-00002896. RW was partially supported by NSF Grant DMS 2247968 and Simons Foundation Gift MP-TSM-00002363}

\begin{abstract}
We prove the rational HK-conjecture for a large class of transformation groupoids in the case when the relevant action has torsion-free stabilizers. A revised version of the rational HK-conjecture in the case of (possibly) torsion stabilizers is introduced and proved for a large class of transformation groupoids. In particular, this revised version holds for Scarparo's counterexamples to the original rational HK-conjecture. The key tools used are the Baum--Connes conjecture and a Chern character defined by Raven.
\end{abstract}

\maketitle

\tableofcontents

\section*{Introduction}
$K$-theory is a powerful invariant for C*-algebras. However, it is often difficult to compute explicitly. Matui's HK-conjecture \cite{Matui2016aa,Matui:2017hq} predicts an isomorphism between the $K$-theory of a sufficiently nice groupoid C*-algebra with the in-principle more computable homology of the groupoid defined by Crainic and Moerdijk \cite{CrainicMoerdijk2000aa}. Despite a number of positive results (see for example \cite{BonickeDellAieraGabeWillett2023aa, ProiettiYasashita2023aa} and references therein), Scarparo gave the first counterexample to the HK conjecture in \cite{Scarparo:2020aa}.  Subsequently, \cite{Deeley2023aa, OrtegaSanchez2022aa, OrtegaScarparo2023aa} exhibited other counterexamples with various additional properties. 

Based on the existence of these counterexamples, one might hope that a weaker version of the HK-conjecture holds. For example, one could hope that the isomorphism holds rationally (i.e., after tensoring with the rational numbers). This version of the conjecture was already considered around the same time as the original conjecture: see for example \cite[Example 3.7]{Matui:2017hq}. However, after considering the existing counterexamples in detail, it becomes clear that this rational version of the conjecture is only reasonable in the case when the groupoid is principal (or some related assumption). This is because, on the one hand, Scarparo's counterexample to the HK-conjecture \cite{Scarparo:2020aa} is already a counterexample to the rational version in the essentially principal case; and on the other hand, as the only known counterexamples in the principal case \cite{Deeley2023aa} do satisfy the rational version. 

The present paper has two main goals. First, we establish the rational version of the HK-conjecture for a large class of transformation groupoids obtained from actions of groups on the Cantor set with torsion-free stabilizers. The basic idea is to use the Baum--Connes conjecture and a Chern character introduced by Raven \cite{Raven:2004aa} to obtain the required isomorphism. The proof actually establishes something more general, which leads to our second goal: a revised version of the rational HK-conjecture. The new conjecture is the same as the old one when the groupoid is principal and ample, but is different when the groupoid has torsion in its isotropy groups; it also makes sense for groupoids without zero-dimensional base space.  Our method proves the revised version of the conjecture for a large class of transformation groupoids.  Again the proof is via the Baum--Connes conjecture and Raven's Chern character.

\subsection*{Detailed statements}

The reader can find more on the definitions and notation used in the main body of the paper; here we just sketch the main inputs.   Given an \'{e}tale groupoid $\mathcal{G}$, we first construct a new \'{e}tale groupoid $\widehat{\mathcal{G}}$ by `blowing-up $\mathcal{G}$ at the torsion elements of its isotropy groups' (this idea has its origins in work of Baum and Connes \cite{Baum:1988qv}). Let $H_{*}(\widehat{\mathcal{G}};\C_{\widehat{\mathcal{G}}^{(0)}})$ denote the Crainic-Moerdijk \cite[Section 3]{CrainicMoerdijk2000aa} homology of the new blow-up groupoid with coefficients in the sheaf $\C_{\widehat{\mathcal{G}}^{(0)}}$ of locally constant functions on the base space.  We also use the notation $H_{**}(\widehat{\mathcal{G}};\C_{\widehat{\mathcal{G}}^{(0)}})$ where in general we use the subscript ``$_{**}$'' to denote the $\Z/2$-graded homology theory whose even (respectively, odd) groups are the direct sum of all of the even (respectively odd) groups of the original $\Z$-graded homology theory.  The homology groups $H_{**}(\widehat{\mathcal{G}};\C_{\widehat{\mathcal{G}}^{(0)}})$ make sense under a weak finite dimensionality assumption on the base space of $\widehat{\mathcal{G}}$ that we denote by ``$c$-$\C$-$\text{dim}(\widehat{\mathcal{G}}^{(0)})<\infty$''; this assumption holds in particular if the original groupoid $\mathcal{G}$ is ample, or if its base space is a manifold.

Here is our reformulated HK conjecture.

\begin{conjecture}\label{hk revised intro}
Suppose that $\mathcal{G}$ is a second countable, locally compact, Hausdorff, \'etale groupoid such that $c$-$\C$-$\text{dim}(\widehat{\mathcal{G}}^{(0)})<\infty$ and such that the rational Baum--Connes conjecture holds for $\mathcal{G}$. Then 
\begin{equation}
K_*(C^*_r(\mathcal{G}))\otimes \C\cong H_{**}(\widehat{\mathcal{G}};\C_{\widehat{\mathcal{G}}^{(0)}}).
\end{equation}
\end{conjecture}
Note that the original rational HK conjecture has several additional assumptions -- the groupoid is minimal, ample, and essentially principal -- that are not relevant for our approach.  On the other hand, the original HK conjecture does not assume the rational Baum--Connes conjecture, and our approach suggests the latter is  important (also see \cite{Proietti:2021wz,BonickeDellAieraGabeWillett2023aa}).

Our main theorem can then be stated as follows.
\begin{theorem}\label{main intro}
Conjecture \ref{hk revised intro} above holds if $\mathcal{G}$ is a transformation groupoid.
\end{theorem}
In particular, this establishes the original rational HK conjecture for transformation groupoids that satisfy the Baum--Connes conjecture, and are associated to actions with torsion-free stabilizers.  It also shows that the natural analog of the rational HK conjecture is true for many non-ample groupoids; results on the HK conjecture for non-ample groupoids have previously been established by Proietti and Yamashita \cite{Proietti:2023aa}.  

The class of transformation groupoids satisfying the Baum--Connes conjecture is large: thanks to deep work of Tu \cite{Tu:1999bq} (based on Higson-Kasparov \cite{Higson:2001eb}) it includes all a-T-menable actions; in particular, this includes all amenable actions, and all actions of a-T-menable (and amenable) groups.

\subsection*{Outline of the proof}

For simplicity, let us consider an action of a torsion-free discrete group $G$ on the Cantor set $X$.  The transformation groupoid is denoted by $G \ltimes X$. We have the following diagram 
\[\xymatrix{
RKK^{G}_*( C_0(\underline{E}G), C(X))\otimes \C \ar[d]^{{\rm ch}} \ar[r]^-{\mu} & K_*(C_r^*(G \ltimes X))\otimes \C\\
H_{**}( G \ltimes X)\otimes \C & }
\]
where 
\begin{enumerate}
\item $\mu$ denotes the Baum--Connes assembly map \cite{Baum:1994pr},
\item $H_{**}( G \ltimes X)$ is the ($\Z/2$-graded variant of the) groupoid homology defined for ample groupoids by Matui \cite{Matui:2017hq} as a simplified version of the Crainic-Moerdijk theory, and
\item ${\rm ch}$ denotes the Chern character defined by Raven \cite{Raven:2004aa}.
\end{enumerate}
Thus, if the rational Baum--Connes assembly map and the Chern character are isomorphisms, then 
\[ K_*(C^*_r(G \ltimes X))\otimes \C \cong H_{**}( G \ltimes X)\otimes \C. \]
Hence the rational HK-conjecture holds for the groupoid $G \ltimes X$. 

Now, the Baum--Connes assembly map being an isomorphism is part of our assumptions, and Raven \cite{Raven:2004aa} established that his Chern character is an isomorphism in general.  Thus for most of the proof of Theorem \ref{main intro}, we do not claim any real originality: the main results we need are due to Raven \cite{Raven:2004aa}, Baum-Schneider \cite{Baum:2002aa} and Schneider \cite{Schneider:1998aa}.  Our goals here are expository: to introduce the ingredients to non-experts, and also to advertise the work of Raven which seems to have been overlooked by experts in this area.  

The only real reason the above sketch is not a complete proof is that the Chern Character that Raven constructed in \cite{Raven:2004aa} does not by definition have codomain in groupoid homology.  The identification of Raven's codomain with groupoid homology is implicit in the work of Baum-Schneider \cite{Baum:2002aa} and Schneider \cite{Schneider:1998aa}, but in a less refined version that we would like, and with limited detail that makes it difficult to access for non-experts.  Most of the work in this paper is a proof that the codomain of Raven's Chern character is naturally isomorphic to groupoid homology: we give a more precise result than can be extracted from the source material in the papers \cite{Baum:2002aa,Schneider:1998aa}, as well as providing far more detail for the benefit of non-experts; this is the main technical innovation of the current paper.

It is worth noting that for our applications to the rational HK-conjecture we could work with the space being totally disconnected and all homology theories having rational or even complex coefficients. Nevertheless, when possible we have worked more generally: as examples of this see the statements of Proposition \ref{bcr and gh} and Theorem \ref{main}.

Based on the structure of the proof of the main result, it is an interesting problem to construct a Chern character for more general groupoids.  Another interesting problem would be to construct a Chern character that works over smaller coefficient rings under appropriate assumptions: the approach in \cite{Luck:2002lk} may be useful here.

\subsection*{Structure of the paper}

Section \ref{secPrelim} contains preliminaries including the precise statements of the HK-conjecture and the rational version. In Section \ref{rcc sec}, Raven's Chern character is introduced. This discussion includes a detailed introduction to the codomain of this Chern character. The explicit isomorphism between this codomain and groupoid homology is stated, but the proof is postponed to Appendix \ref{hh app}; this is because we wanted to make the main body of the paper available to readers who prefer to treat the result of Appendix \ref{hh app} as a `black box'. Section \ref{rcc sec} contains the main results: these are Theorem \ref{main} and Corollary \ref{main cor}. In addition, the revised version of the rational HK-conjecture (Conjecture \ref{hk revised intro} above) is carefully stated for groupoids with torsion stabilizers (see Conjecture \ref{hk revised}), and a few examples and computational techniques are discussed. 

\section*{Acknowledgments}
Both authors thank the University of Colorado Boulder and the University of \Hawaii\ at \Manoa\ for facilitating this collaboration. The first listed author thanks the Fields Institute for support during a visit for the Thematic Program on Operator Algebras and Applications in the fall of 2023. In particular, this visit to the Fields facilitated useful discussions with Alistair Miller and Ryszard Nest.  We would also like to thank Wolfgang L\"{u}ck, Alistair Miller, Valerio Proietti, and Christian Voigt for useful comments on a preliminary draft.

\section{Preliminaries} \label{secPrelim}

We start by introducing notational conventions for \'{e}tale groupoids and their $C^*$-algebras: see for example \cite{Sims:2017aa} for background material.  Let $\mathcal{G}$ be a groupoid with unit space $\mathcal{G}^{(0)}$ and range/source maps denoted by $r, s : \mathcal{G} \to \mathcal{G}^{(0)}$. The ordered pair $g, h \in \mathcal{G}$ is composable if $s(g) = r(h)$; their composition is denoted $gh$.  The inverse of $g \in \mathcal{G}$ is denoted $g^{-1}$. All groupoids considered will be locally compact and Hausdorff. Moreover, all groupoids in the paper will be \'etale, meaning that $r$ and $s$ are local homeomorphisms. In this case $\mathcal{G}^{(0)}$ is an open subset of $\mathcal{G}$ and the relevant Haar system is given by counting measures. We say that $\mathcal{G}$ is principal if for each $x \in \mathcal{G}^{(0)}$ the isotropy group
\[ \mathcal{G}^x_x := \{ g \in \mathcal{G} \mid s(g) = r(g) = x \} \] 
is trivial (i.e., equal to $\{ x \}$). We say that $\mathcal{G}$ is essentially principal if the interior of the set $\{ g \in \mathcal{G} \mid s(g)=r(g) \}$ is $\mathcal{G}^{(0)}$. Notice that principal implies essentially principal, but the converse is false. A groupoid is ample if its unit space is totally disconnected (e.g., the Cantor set).

To a groupoid $\mathcal{G}$ satisfying the assumptions above one can associate its reduced groupoid $C^*$-algebra. The resulting $C^*$-algebra is denoted by $C^*_r(\mathcal{G})$. The computation of the $K$-theory of $C^*_r(\mathcal{G})$ is an important problem. The homology of $\mathcal{G}$ was defined in \cite{CrainicMoerdijk2000aa} and will be denoted by $H_*(\mathcal{G})$. 

With this notation introduced, Matui's HK-conjecture \cite{Matui2016aa} is the following:
\begin{conjecture}
Suppose that $\mathcal{G}$ is a second countable, \'etale, essentially principal, minimal, ample groupoid. Then
\[ K_*(C^*_r(\mathcal{G})) \cong H_{**}(\mathcal{G}). \] 
\end{conjecture}
(Here we use the convention from the introduction that $H_{**}$ denotes the $\Z/2$-graded homology theory associated to a $\Z$-graded homology theory $H_*$ defined to have even (respectively, odd) group the direct sum of all the even (odd) groups of $H_*$).  The rational version of this conjecture is the following:
\begin{conjecture}
Suppose that $\mathcal{G}$ is a second countable, \'etale, essentially principal, minimal, ample groupoid. Then
\[ K_*(C^*_r(\mathcal{G})) \otimes \Q \cong H_{**}(\mathcal{G})\otimes \Q .\] 
\end{conjecture}
There are counterexamples to both these conjectures. The first counterexample is due to Scarparo \cite{Scarparo:2020aa}, also see \cite{Deeley2023aa, OrtegaScarparo2023aa, OrtegaSanchez2022aa} for other counterexamples. However, currently there is no counterexample to the {\bf rational} version of the conjecture when $\mathcal{G}$ is principal (rather than just essentially principal).

Most of the groupoids considered in this paper are constructed from group actions.   Throughout the paper, $G$ denotes a discrete group.  A \emph{$G$-space} is a topological space equipped with a left action of a discrete group $G$ by homeomorphisms.  Let $X$ be a locally compact Hausdorff $G$-space. The transformation groupoid $\mathcal{G}=G \ltimes X $ associated to this data is defined to be $G \times X$ where
\begin{enumerate}[(a)]
\item The composition $(\gamma, x) \cdot (\alpha, y)$ is equal to $(\gamma \alpha, y)$ when $x=\alpha y$ and is not defined otherwise.
\item The inverse of $(\gamma, x)$ is given by $(\gamma^{-1}, \gamma x)$.
\item Upon identifying $\mathcal{G}^{(0)}$ with $X$ via $x\in X \mapsto (e, x)\in \mathcal{G}$, the range and source maps are given by $(\gamma, x) \mapsto \gamma x$ and $(\gamma, x) \mapsto x$ respectively.
\end{enumerate}
In relation to the HK-conjecture, we have the following relationships between properties of the action of $G$ on $X$ and properties of the transformation groupoid $G \ltimes X$:  
\begin{enumerate}[(a)]
\item The transformation groupoid is second countable if $X$ is second countable and $G$ is countable.
\item The transformation groupoid is \'etale because $G$ is discrete. 
\item The transformation groupoid is (essentially) principal if and only if the action of $G$ on $X$ is (topologically) free.
\item The transformation groupoid is minimal if and only if the action of $G$ on $X$ is minimal.
\item The transformation groupoid is ample if and only if $X$ is totally disconnected.
\end{enumerate}

We will also need to consider groupoid actions.  For the following definition, if $X$, $Y$, and $Z$ are topological spaces equipped with continuous maps $f:X\to Z$ and $g:Y\to Z$, then we define $X_f\! \times_{g} Y:=\{(x,y)\in X\times Y\mid f(x)=g(y)\}$ equipped with the topology it inherits from $X\times Y$.
\begin{definition} \label{def:GroupoidAction}
Suppose $\mathcal{G}$ is a groupoid and $Y$ is space. Then an \emph{action} of $\mathcal{G}$ on $Y$ is given by the following data. There are maps $p: Y \rightarrow \mathcal{G}^{(0)}$ and $\mathcal{G}_s \!\times_{p} Y \rightarrow Y$ written $(g, y) \mapsto gy$ such that 
\begin{enumerate}[(a)]
\item $p(g y)= p(y)$;
\item $p(y) y = y$;
\item $g( hy) = (gh) y$.
\end{enumerate}
The map $p$ is called the \emph{anchor map}.
\end{definition}

Throughout, $k$ denotes a unital commutative ring.  If $S$ is a topological space, then $k[S]$ denotes the ring of compactly supported and locally constant functions from $S$ to $k$ (if there is no given topology on $S$ then we use the discrete topology, in which case the `locally constant' condition is automatic).  If $S$ is equipped with a $G$-action by homeomorphisms, then $k[S]$ is equipped with the induced action.   We will also write $kG$ for the group ring of $G$ equipped with the usual convolution multiplication: this should be contrasted with $k[G]$, which for us means the ring of compactly supported $k$-valued function on $H$ equipped with \emph{pointwise} multiplication.

\section{Baum-Schneider homology and Raven's Chern character }\label{rcc sec}

\subsection{Derived categories and Baum-Schneider bivariant homology}

We will need to use derived categories of abelian categories, and derived functors between them.  Here we introduce the basic ideas and notation; see for example \cite[Chapter 10]{Weibel:1995ty} or \cite[Chapter 1]{Kashiwara:1990aa} for more detailed background.  

Let $\mathcal{A}$ be an abelian category (we discuss examples below).  We will write $\dcc(\mathcal{A})$ for the derived category of bounded below cochain complexes from $\mathcal{A}$.  Thus objects of $\dcc(\mathcal{A})$ are cochain complexes
$$
\cdots \to A_{-2}\to A_{-1}\to A_0\to A_1\to A_2 \to \cdots 
$$
of objects from $\mathcal{A}$ indexed by $\Z$, and where $A_n=0$ for all $n$ suitably small.  Morphisms in $\dcc(\mathcal{A})$ are based on chain morphisms, but with all quasi-isomorphisms formally inverted: see the background references above for details.   

Compare \cite[Definition 10.7.1]{Weibel:1995ty} for the following definition.

\begin{definition}\label{hext}
Let $\mathcal{A}$ be an abelian category, $n\in \Z$, and let $A$ and $B$ be objects of $\dcc(\mathcal{A})$.  Let  ``$[-n]$'' denote the operation of shifting a chain complex `right' by $n$, so that the degree zero part of $B[-n]$ is the degree $-n$ part of $B$ and so on.  The \emph{hyperext}\footnote{If $A$ and $B$ are objects of $\mathcal{A}$ identified with cochain complexes concentrated in degree zero, and if $\mathcal{A}$ has enough projectives, then $\text{Ext}_{\mathcal{A}}^n(A,B)$ is the usual $\text{Ext}$ group.} group of $A$ and $B$ is defined to be 
$$
\text{Ext}_{\mathcal{A}}^n(A,B):=\text{Hom}_{\dcc(\mathcal{A})}(A,B[-n]).
$$
\end{definition}

\begin{examples}
Here are the key examples of abelian categories we will use.  For the first example, we assume the reader is (somewhat) familiar with the basic definitions of sheaf theory: background on this can be found (for example) in \cite{Bredon:1997aa}, \cite{Godement:1958aa}, \cite{Swan:1964aa}, or \cite[Chapter 2]{Kashiwara:1990aa}.  
\begin{itemize}
\item Let $Z$ be a locally compact, Hausdorff topological $G$-space.  Let $k_Z$ denote the sheaf of locally constant functions from $Z$ to $k$.  A \emph{sheaf of $k$-modules} over $Z$ is then a sheaf $F$ where there is a continuous module action of $k_Z$ in the natural sense: compare for example \cite[page 3]{Bredon:1997aa}.  Following \cite[page 195]{Grothendieck:1957aa} a \emph{$G$-sheaf} $F$ of $k$-modules over $Z$ is a sheaf of $k$-modules over $Z$ such that if $\pi:\mathcal{F}\to Z$ is the corresponding \'{e}tale space\footnote{i.e.\ the space defining a sheaf in the form appearing for example in \cite[Definition 1.2]{Bredon:1997aa}.}, then $\mathcal{F}$ is a $G$-space and $\pi$ is equivariant\footnote{See also \cite[Definitions 6.1.2 and 6.1.5]{Raven:2004aa} for an equivalent definition of $G$-sheaves in terms of presheaves.}.  Morphisms between $G$-sheaves (of $k$-modules) over $Z$ are defined to be sheaf (module) homomorphisms such that the induced maps between \'{e}tale spaces are equivariant.  With these morphisms, the $G$-sheaves of $k$-modules over $Z$ form an abelian category that we denote $\text{Sh}_G(Z)$: kernels and cokernels have the same underlying sheaves as in the non-equivariant case, equipped with the induced $G$-action. 
\item Let $R$ be a (possibly non-unital) ring equipped with an action of $G$ by automorphisms.  A module $M$ over $R$ is \emph{non-degenerate} if $RM=M$.  We write $G$-$R$ for the category of nondegenerate $R$-modules equipped with a compatible\footnote{Precisely, if $\rho:G\to \text{Aut}(R)$ is the $G$-action, then there is an action $\mu:G\to \text{Aut}_{Ab}(M)$ by abelian group automorphisms (\emph{not} $R$-module automorphisms!) that satisfies $\mu_g(rm)=\rho_g(r)\mu_g(m)$ for all $r\in R$, $m\in M$, and $g\in G$.} $G$-action.  The key examples of $R$ we will use are $R=k$ equipped with the trivial $G$-action, and $R=k[G_{tor}]$ the ring of finitely supported functions from $G_{tor}$ to $k$ with pointwise operations, where $G_{tor}$ is the set of torsion elements of $G$ equipped with the conjugation $G$-action and $k[G_{tor}]$ is equipped with the induced action.  
\end{itemize}
\end{examples}

The following definition is based on \cite[Line (12.2) on page 202]{Baum:1988qv}.

\begin{definition}\label{z hat}
Let $G_{tor}$ be the subset of torsion elements of $G$, equipped with the conjugation action, and let $Z$ be a $G$-space.  Define 
$$
\widehat{Z}:=\{(z,g)\in Z\times G_{tor}\mid gz=z\}
$$
equipped with the subspace topology and $G$-action that it inherits from $Z\times G_{tor}$.  We write $\phi_Z:\widehat{Z}\to G_{tor}$ for the (restriction of the) canonical equivariant projection to the second variable.  
\end{definition}

\begin{definition}\label{gammac}
Let $Z$ and $\widehat{Z}$ be as in Definition \ref{z hat} above, and let $F$ be a $G$-sheaf of $k$-modules on $\widehat{Z}$.  Let $\Gamma_c(\widehat{Z},F)$ be the module of compactly supported sections; this is a $G$-$k[G_{tor}]$ module for the induced $G$-action, and the $k[G_{tor}]$-action defined using $\phi_Z$.   The process of taking compactly supported sections thus defines a functor
$$
\Gamma_c:\text{Sh}_G(\widehat{Z})\to G\text{-}k[G_{tor}],\quad F\mapsto \Gamma_c(\widehat{Z},F)
$$
from the abelian category of $G$-equivariant sheaves of $k$-modules over $\widehat{Z}$ to the abelian category of $G$-$k[G_{tor}]$ modules. 
\end{definition}

We will need to also discuss derived functors between derived categories.  Let $F:\mathcal{A}\to \mathcal{B}$ be an additive functor between abelian categories.  The canonical `term-wise' extension of $F$ to chain complexes does not in general define a functor $\dcc(\mathcal{A})\to \dcc(\mathcal{B})$ as $F$ does not typically take quasi-isomorphisms to something invertible.  Roughly speaking, the (total) right derived functor\footnote{The total right derived functor should not be confused with the more classical right derived functors (or `satellite functors') $R^iF$: indeed, $R^iF(A)$ is isomorphic to the $i^\text{th}$ cohomology group of $RF(A)$, and thus the groups $R^iF$ contain rather less information than $RF$.} 
$$
RF:\dcc(\mathcal{A})\to \dcc(\mathcal{B})
$$
is the `best approximation' to $F$ that does make sense on the level of derived categories.  One way to compute $RF(A)$ for some cochain complex $A\in \dcc(\mathcal{A})$ is to replace $A$ with a quasi-isomorphic cochain complex $I$ consisting of injective objects from $\mathcal{A}$ (this can be done if $\mathcal{A}$ has enough injectives), and then apply $F$ term-wise to $I$: see for example \cite[Section 10.5]{Weibel:1995ty} for details.  There are also left-derived functors, which are analogous but with appropriate arrows reversed (and in particular with the role of injectives instead played by projectives).  

Now, just as for the classical case of non-equivariant sheaves, the functor $\Gamma_c$  of Definition \ref{gammac} is left exact and the category $\text{Sh}_G(\widehat{Z})$ of $G$-sheaves over $\widehat{Z}$ has enough injective objects (see \cite[Proposition 5.1.2]{Grothendieck:1957aa} or \cite[Lemma 1]{Schneider:1998aa}).  Hence the total right derived functor of $\Gamma_c$ exists: see for example \cite[Existence Theorem 10.5.6]{Weibel:1995ty}.

\begin{definition}\label{r der}
We write 
$$
R\Gamma_c:\dcc(\text{Sh}_G(\widehat{Z}))\to \dcc(G\text{-}k[G_{tor}])
$$
for the total right derived functor of the functor $\Gamma_c$ from Definition \ref{gammac}.
\end{definition}

The following definition is taken from \cite[pages 315-6]{Baum:2002aa}.  Recall that we write $k_Z$ for the $G$-sheaf of locally constant $k$-valued functions on a $G$-space $Z$.

\begin{definition}\label{bs hom}
Let $Z$ and $X$ be locally compact, Hausdorff $G$-spaces.  The \emph{Baum-Schneider homology} $H^*_{G,c,k}(Z,X)$ is defined to be the hyperext group (see Definition \ref{hext})
$$
H^*_{G,k,c}(Z,X):=\text{Ext}_{G\text{-}k[G_{tor}]}^*(R\Gamma_ck_{\widehat{Z}},R\Gamma_ck_{\widehat{X}})
$$
of $R\Gamma_ck_{\widehat{Z}}$ and $R\Gamma_ck_{\widehat{X}}$ (see Definition \ref{r der}). 
\end{definition}

Our goal in the remainder of this subsection is to briefly describe another bivariant cohomology theory introduced by Raven, and show that it is the same as the Baum-Schneider theory from Definition \ref{bs hom}.  This is no doubt known to experts, but does not appear to be explicitly recorded.

The following are taken from \cite[Definitions 6.5.1 and 7.3.5]{Raven:2004aa}. 

\begin{definition}\label{rav hom}
Let $Z$ and $X$ be locally compact, Hausdorff $G$-spaces.  Define 
$$
HH^*_{G,k}(Z,X):=\text{Ext}^*_{G\text{-}k}(R\Gamma_ck_Z,R\Gamma_ck_X),
$$
where $\Gamma_c:\text{Sh}_G(Z)\to G\text{-}k$ is the (more basic) analogue of the functor from Definition \ref{r der} without $Z$ replaced by $\widehat{Z}$.  

As usual, let $G_{tor}$ denote the subset of $G$ consisting of torsion elements, and let $G_{tor}//G$ denote the quotient of $G_{tor}$ by the conjugation action of $G$; for each class $c\in G_{tor}//G$ fix a representative $g_c$. For each $g\in G$, let $Z(g)$ denote the centralizer of $G$, and let $X^g$ denote the $g$-fixed points.  Define
$$
\widehat{HH}^*_{G,k} (Z,X):=\bigoplus_{c\in G_{tor}//G}  HH^*_{Z(g_c),k}(Z^{g_c},X^{g_c})
$$
\end{definition}

We now show that the groups from Definition \ref{bs hom} and \ref{rav hom} are the same.

\begin{lemma}\label{bs=r}
Let $X$ and $Y$ be locally compact and Hausdorff $G$-spaces.  Then there is a canonical identification  
$$
\widehat{HH}^*_{G,k}(Y,X)\cong H^*_{G,k,c}(Y,X)
$$
of bifunctors.
\end{lemma}

\begin{proof}
We claim that there is an identification of categories 
$$
G\text{-}k[G_{tor}]\cong \bigoplus_{c\in G_{tor}//G}(Z(g_c)\text{-}k[c])
$$
where the right hand side has objects (resp.\ morphisms) given by direct sums of objects (resp.\ morphisms), one from each category $Z(g_c)\text{-}k[c]$.  More explicitly, this isomorphism is given by the map defined on objects by
\begin{equation}\label{split map}
M\mapsto (\chi_{g_c}\cdot M)_{c\in G_{tor}//G}
\end{equation}
where $\chi_{g_c}\in k[c]$ is the characteristic function of the singleton $\{g_c\}\subseteq c$, and $Z(g_c)$ acts on $\chi_{g_c}\cdot M$ via the restriction of the original $G$-action; the definition on morphisms is the canonical one compatible with this.  Indeed, an inverse to the map in line \eqref{split map} is given on objects by  
$$
(M_c)_{c\in G_{tor}//G}\mapsto \bigoplus_{c\in G_{tor}//G} \text{Map}(G,M_{c})^{Z(g_c)},
$$
where the summand $\text{Map}(G,M_{c})^{Z(g_c)}$ on the right hand-side means the $Z(g_c)$-equivariant maps from $G$ to $M_{c}$ that are supported on finitely many $Z(g_c)$-cosets; we leave this to the reader to check.  Having established the claim, the lemma follows from the definition of the hyperext groups on noting that there are canonical isomorphisms $\chi_{g_c}\cdot \Gamma_c(k_{\widehat{X}})\cong\Gamma_c(k_{X^{g_c}})$.  
\end{proof}

\subsection{Isomorphism between homologies}

We now move on to the second main definition needed for the main result.  

Write $\dch(\mathcal{A})$ for the derived category of bounded below chain complexes over an abelian category $\mathcal{A}$.  

\begin{definition}\label{coinv}
Let $G$-$\Z$ denote the abelian category of $G$-$\Z$ modules, and $\Z\text{-mod}$ the abelian category of abelian groups (i.e.\ modules over $\Z$).  Let 
$$
(\cdot)_G:G\text{-mod}\to \text{$\Z$-mod},\quad A\mapsto A_G
$$ 
be the functor taking $A$ to the group of coinvariants, i.e.\ to $A_G:=A\otimes_{\Z G}\Z$.
\end{definition}

The category $G$-$\Z$ has enough projective objects (as it is the same as the category of modules over the group ring $\Z G$).  Hence the total left derived functor of $(\cdot)_G$ exists: see for example \cite[Existence Theorem 10.5.6]{Weibel:1995ty}\footnote{The reference states the existence result for left-derived functors for bounded \emph{above cochain} complexes; we are using a slight variant for bounded \emph{below chain} complexes here.}.

\begin{definition}
We write 
$$
L_G:\dch(G\text{-mod})\to \dch(\Z\text{-mod})
$$
for the total left derived functor of the coinvariant functor of Definition \ref{coinv}.
\end{definition}

Here is the key definition of this subsection.

\begin{definition}\label{ghh}
Let $A$ be a bounded below chain complex of $G$-modules.  The \emph{hyperhomology} of $G$ with coefficients in $A$ is define by 
$$
H_n(G,A):=H_n(L_G(A)).
$$
\end{definition}

\begin{remark}\label{gkmodrem}
If $k$ is a commutative unital ring, we may also consider $L_G:\dch(G\text{-}k)\to \dch(k\text{-mod})$ analogously.  The underlying hyperhomology groups have the same underlying abelian groups as in the case above: the key point is that if $A$ is a $G$-$k$ module, then $A\otimes_{kG}k\cong A\otimes_{\Z G} \Z$ as abelian groups.  Hence whichever convention we use makes no real difference; we will tend to work with $G$-$k$-modules as this is slightly more convenient for some arguments.
\end{remark}

See for example \cite[Corollary 10.5.7]{Weibel:1995ty} for a proof that Definition \ref{ghh} with the more classical notion of hyperhomology\footnote{The result of \cite[Corollary 10.5.7]{Weibel:1995ty} specializes to $H_n(G,A):=H_{-n}(L_G(A))$; we do not have the minus sign as \cite[Chapter 10]{Weibel:1995ty} works with cochain complexes and we are working with chain complexes.  Similarly, \cite[Corollary 10.5.7]{Weibel:1995ty} works with bounded above complexes, and we are working with bounded below complexes.}; the latter can be found in \cite[Example 6.1.15]{Weibel:1995ty}.

We need to recall two definitions from sheaf theory: compare \cite[Definition 2.5.5 and Proposition 2.5.6]{Kashiwara:1990aa}.

\begin{definition}\label{c-soft}
Let $Z$ be a locally compact, Hausdorff topological space.  A sheaf $F$ over $Z$ is \emph{$c$-soft} if for any compact subset $K$ of $Z$, the restriction map 
$$
\Gamma_c(Z,F)\to \lim_{U\supseteq K}\Gamma_c(U,F)
$$
(here the direct limit is taken over all open neighbourhoods $U$ of $K$) is surjective.   A $G$-sheaf over a locally compact, Hausdorff $G$-space is \emph{$c$-soft} if the underlying sheaf is $c$-soft.
\end{definition}

The following definition is based on \cite[Definition 16.3 and Theorem 16.4]{Bredon:1997aa}.

\begin{definition}\label{cdim}
Let $Z$ be a locally compact Hausdorff topological space and let $k$ be a unital commutative ring.  The \emph{$c$-$k$-cohomological dimension} of $Z$ is the smallest integer $m$ such that every sheaf of $k$-modules on $Z$ admits a $c$-soft resolution of length at most $m$ (and infinity if no such $m$ exists).  We write $c$-$k$-$\text{dim}(Z)$ for this dimension.
\end{definition}

We will discuss some examples of this in Remark \ref{dim rem} below.

Finally, we need one more definition due to Baum and Schneider \cite[page 316]{Baum:2002aa}.

\begin{definition}\label{bs!}
Let $X$ be a locally compact, Hausdorff $G$-space, and let $Z$ be a proper Hausdorff $G$-space (not necessarily locally compact).  Define
$$
H^*_{G,k,!}(Z,X):=\lim_{Y\subseteq X}H_{G,k,c}^*(Y,X)
$$
where the groups on the right are as in Definition \ref{bs hom}, and the limit is taken over all $G$-compact (therefore locally compact) $G$-invariant subspaces $Y$ of $Z$.
\end{definition}

Here is the main result of this subsection.  The special case that $k=\C$ follows from combining work of Schneider \cite[Section 4]{Schneider:1998aa} and Baum-Schneider \cite[Section 1.B]{Baum:2002aa}

\begin{proposition}\label{bcr and gh}
Let $X$ be a locally compact, Hausdorff $G$-space with finite $c$-$k$-cohomological dimension, and let $\widehat{X}$ be as in Definition \ref{z hat}.  Let 
$$
0\to k_{\widehat{X}}\to I^0\to I^1\to \cdots \to I^m\to 0
$$
be a finite-length resolution of $k_{\widehat{X}}$ by $c$-soft $G$-sheaves of $k$-modules\footnote{Such a resolution always exists: see Remark \ref{dim rem} below.}, and let $\Gamma_c(I^{-\bullet})$ denote the induced chain complex 
$$
\Gamma_c(\widehat{X},I^m)\leftarrow \cdots \leftarrow \Gamma_c(\widehat{X},I^0)
$$
of $G$-modules, where $\Gamma_c(\widehat{X},I^m)$ appears in degree $-m$ and $\Gamma_c(\widehat{X},I^0)$ in degree zero.  Let $k$ be a Pr\"{u}fer domain\footnote{See Remark \ref{k rem} below for more about what this assumptions means, and examples.} such that $n^{-1}$ exists in $k$ whenever there is an order $n$ stabilizer of a point $x\in X$.  Let $\underline{E}G$ be the universal $G$-space for proper actions.  

Then for each $n\in \Z$, there is a canonical isomorphism
\begin{equation}\label{bsr iso}
H^n_{G,k,!}(\underline{E}G,X)\cong H_{-n}(G,\Gamma_c(I^{-\bullet}))
\end{equation}
between the Baum-Schneider bivariant cohomology groups and the group hyperhomology with coefficients in the $G$-$k$ chain complex $\Gamma_c(I^{-\bullet})$.
\end{proposition}

\begin{remark}\label{k rem}
Let us say a little about the assumptions on $k$ in Proposition \ref{bcr and gh}.  A Pr\"{u}fer domain is a (commutative, unital) integral domain such that all torsion free modules are flat.  The most important examples for us include fields (such as $\Q$ and $\C$), and PIDs such as $\Z$ and its localizations (such as $\Z[1/2]$).  

Schneider \cite{Schneider:1998aa} and Baum-Schneider \cite{Baum:2002aa} work only with $k=\C$ (but point out that their results really just need a field of characteristic zero).  It seems useful to consider other rings; however, the most interesting case is $k=\Z$, which satisfies the assumptions of Proposition \ref{bcr and gh} if (and only if) the stabilizers for the $G$-action on $X$ are all torsion free, and therefore in particular for all actions on appropriate spaces when $G$ is torsion-free. 

On the other hand, the isomorphism in line \eqref{bsr iso} can fail for $k=\Z$ if there are torsion stabilizers.  Indeed, let $G=C_m$ be a finite cyclic group of order $m$, $k=\Z$, and $X$ be a single point.  First, note that $\widehat{X}=G_{tor}=G$, with the $G$ action being by conjugation (so trivial, as $G$ is abelian).   Then $k_{\widehat{X}}$ is itself $c$-soft, and $\Gamma_c(\widehat{X},k_{\widehat{X}})$ is just $\Z^{|G|}$ with the trivial $G$-action.  Hence 
\begin{equation}\label{gp hom}
H_n(G,\Gamma_c(\widehat{X},k_{\widehat{X}}))\cong H_n(G,\Z)^{|G|}\cong \left\{\begin{array}{ll} \Z & n=0 \\ (\Z/m)^m & n>0 \text{ even} \\ 0 & \text{otherwise} \end{array}\right.
\end{equation}
(see for example \cite[page 35]{Brow:1982rt} for a computation of the group homology of $C_m$).  On the other hand, for finite $G$, $\underline{E}G$ can be taken to be a single point, and the isomorphism in Lemma \ref{bs=r} specializes to
$$
H^*_{G,k,!}(\underline{E}G,X)=\bigoplus_{G_{tor}} \text{Ext}^*_{G\text{-}k}(R\Gamma_ck_{pt},R\Gamma_c k_{pt}).
$$
For a single point the sheaf $k_{pt}$ is $c$-soft, whence by Lemma \ref{f-inj} below $R\Gamma_ck_{pt}=k$, where here ``$k$'' is interpreted as an element of the derived category of $G$-$k$ modules, i.e.\ the chain complex of $G$-$k$ modules that equals $k$ in degree zero and zero in all other degrees, equipped with the trivial $G$-action.  Note that for $k=\Z$, $\text{Ext}^*_{G\text{-}k}(k,k)$ is (essentially by definition -  see for example \cite[Section III.2]{Brow:1982rt}) the group cohomology $H^*(G,\Z)$ of $G$.  Hence 
\begin{equation}\label{gp coh}
H^n_{G,!}(\underline{E}G,X)=\bigoplus_{G_{tor}} \text{Ext}_{G\text{-}k}^n(k,k)=H^n(G,\Z)^{|G|}\cong  \left\{\begin{array}{ll} \Z & n=0 \\  (\Z/m)^m  & n>0 \text{ odd} \\ 0 &  \text{otherwise}  \end{array}\right.
\end{equation}
(see for example \cite[Example 2 on page 58]{Brow:1982rt} for a computation of the group cohomology of $C_m$ in terms of the group homology).  The computations in lines \eqref{gp hom} and \eqref{gp coh} show that the isomorphism in line \eqref{bsr iso} does not hold in this case.  Note however that it is true if $k=\Z[1/m]$ (as predicted by the proposition!), as this has the effect of tensoring all the groups in lines \eqref{gp hom} and \eqref{gp coh} by $\Z[1/m]$, and thus all the $m$-torsion vanishes.
\end{remark}

\begin{remark}\label{dim rem}
Let us also say a little about the assumption that $X$ has finite $c$-$k$-cohomological dimension, which says that any sheaf on $X$ admits a finite-length resolution by $c$-soft sheaves.  We show in Lemma \ref{rgcx} below that finite $c$-$k$-cohomological dimension implies the existence of a finite length resolution by $c$-soft $G$-sheaves as needed for the statement of the proposition (compare also \cite[3.3]{CrainicMoerdijk2000aa}).  We make this assumption as we work in the framework of derived categories and derived functors, and this works best\footnote{It is, however, conceivable to us that the more classical approach to hyperhomology as exposited in \cite[Chapter XVII]{Cartan:1956aa} or \cite[Section 5.7]{Weibel:1995ty} would enable the results to be carried through without this assumption; we did not seriously pursue that here.} with bounded below complexes (see for example \cite[Section 10.5, particularly Corollary 10.5.7]{Weibel:1995ty}).

Let us say also a little about geometric assumptions implying finiteness of $c$-$k$-cohomological dimension.  Recall that the \emph{covering dimension} of a topological space $X$ is the smallest natural number $m$ such that every open cover of $X$ admits a refinement where at most $m+1$ sets intersect (and infinity if no such $m$ exists).  Let us write ``$c$-$\text{dim}(X)\leq m$'' if the covering dimension of every compact subset of $X$ is at most $m$\footnote{This is strictly weaker than having covering dimension at most $m$ for general locally compact Hausdorff spaces: for example the \emph{long line} of \cite[Counterexample 46]{Steen:1970aa} satisfies $c$-$\text{dim}(X)=1$, but is not paracompact, so has infinite covering dimension.  It is also strictly weaker than the corresponding notions one gets by considering the so-called small or large inductive dimension of compact subspaces: see \cite[Theorems 3.1.28 and 3.1.29, and Example 3.1.31]{Engelking:1978aa}.  On the other hand, for `reasonable' spaces, all these notions of dimension coincide: compare for example \cite[Theorem 1.7.7]{Engelking:1978aa}.}.  Then thanks to the isomorphism in \cite[Corollary III.4.12]{Bredon:1997aa}, we see that the sheaf cohomology of every sheaf on every compact subset of $X$ vanishes in dimensions above $m$.  Using \cite[Theorem II.16.4 and Proposition II.16.7]{Bredon:1997aa}, this implies that the $c$-$k$-cohomological dimension of $X$ is at most $m$.  

If we put more assumptions on $X$, then we can say more.  For example, assume that $X$ is a smooth $m$-manifold with a smooth $G$-action.  This implies that $c$-$\text{dim}(X)\leq m$.  Note moreover that $\widehat{X}$ is also a manifold (with components of possibly varying dimension).  We may take $I^\bullet:=\Omega^\bullet_{\widehat{X}}$ to be the complex of smooth differential forms on $\widehat{X}$, equipped with the de Rham differential and induced $G$-action.  This is then a resolution of $k_{\widehat{X}}$ of length at most $m$ by $c$-soft $G$-sheaves.

As another especially interesting case, note that $c$-$\text{dim}(X)=0$ for a locally compact Hausdorff space $X$ if and only if the topology on $X$ has a basis of compact open sets\footnote{This is again strictly weaker than having covering dimension zero in general: \cite[Counterexample 65]{Steen:1970aa} is (locally compact, Hausdorff) and has a basis of compact open sets, but does not have covering dimension zero (in fact, that space cannot be written as a disjoint union of compact open sets at all).}.  This is equivalent to various other notions of `total disconnectedness' such as the condition that all connected components are single points: see for example \cite[Theorem 1.4.5]{Engelking:1978aa}.  One can then check directly that $k_{\widehat{X}}$ is $c$-soft itself (compare \cite[Proposition 2.8]{Proietti:2021wz}), or appeal to \cite[Corollary II.16.38]{Bredon:1997aa} for this, as above.  Hence in the zero-dimensional case, Proposition \ref{bcr and gh} reduces to an isomorphism 
$$
H^n_{G,k,!}(\underline{E}G,X)\cong H_{-n}(G,\Gamma_c(\widehat{X},k_{\widehat{X}})),
$$ 
where the right hand side is just the usual (non-hyper!)\ group homology with coefficients in the $G$-module of compactly supported and continuous functions from $\widehat{X}$ to $k$.  If moreover all stabilizers of points in $X$ are torsion free, then the right hand side becomes $H_{-n}(G,\Gamma_c(X,k_X))$ and (up to multiplying degrees by $-1$) this agrees with the usual groupoid homology of the transformation groupoid $G\ltimes X$ with coefficients in $k$ as appearing in the HK conjecture.
\end{remark}

The proof of Proposition \ref{bcr and gh} is rather long.  Moreover, it is entirely within the realm of derived categories and equivariant sheaf theory, and we imagine most of our potential readers have background in $C^*$-algebra $K$-theory and may be happier treating Proposition \ref{bcr and gh} as a black box.  For these two reasons, we postpone the argument to Appendix \ref{hh app}.

\subsection{Raven's Chern character}

The following result is \cite[Corollary 7.3.12]{Raven:2004aa}.  

\begin{theorem}[Raven]\label{raven chern}
Let $G$ be a countable discrete group, let $X$ be a second countable, locally compact $G$-space, let $Y$ be a $G$-finite and proper $G$-CW complex, and let $\widehat{HH}_{G,k}^{*}(Y,X)$ be as in Definition \ref{rav hom}.  Then there is a ``Chern-Raven'' character  
$$
ch_R:KK^G_*(C_0(Y),C_0(X))\otimes \C\to \widehat{HH}_{G,k}^{**}(Y,X)
$$ 
that is natural for equivariant proper maps in either variable, and an isomorphism.   \qed
\end{theorem}

\begin{remark}\label{cc rems}
Raven's proof of Theorem \ref{raven chern} proceeds by constructing an isomorphism between $KK^G$-theory and Baum-Douglas topological $KK^G$-theory (denoted $tKK^G$ by Raven - see \cite[Definition 4.3.1]{Raven:2004aa}): see \cite[Corollary 4.7.9]{Raven:2004aa}\footnote{Isomorphisms between Baum-Douglas and Kasparov models for $K$-homology and more generally $KK$-theory are also considered in \cite{Baum:2009hq, BaumOyonoSchickWalter2010aa, EmersonMeyer2010aa, Jakob1998aa, Jakob2000aa}.  Raven's version is the only one that is strong enough to be used directly for our purposes, but it seems that with some work the results in \cite{EmersonMeyer2010aa} could be used.}.  Cycles for the latter theory are given by $G$-spin$^{\text{c}}$ manifolds with some additional data, and  Raven is then able to use versions of the classical $K$-homology Chern character for spin$^{\text{c}}$ manifolds to build the map $ch_R$ from Theorem \ref{raven chern}.  Let us make some remarks on the connection of this result to other (bivariant, equivariant) Chern characters in the literature.   

First, Baum and Schneider \cite[Section 3, Corollary 5]{Baum:2002aa} establish an analogue of Theorem \ref{raven chern} where $G$ is a profinite group.  The proof is quite different to Raven's: it first establishes an isomorphism 
$$
KK^G_*(C_0(Y),C_0(X))\cong \text{Hom}_{G\text{-}\C[G]}(K^*(\widehat{Y})\otimes \C,K^*(\widehat{X})\otimes \C)
$$
for profinite $G$, where the $G$-action of $\C[G]$ is by conjugation (see \cite[Section 3, Proposition 4]{Baum:2002aa}).  Baum and Schneider then use the classical Chern character isomorphism $K^*(\widehat{X})\otimes \C\cong H^{**}(\widehat{X},\C)$ to identify the right hand side with an appropriate hom-group with $K$-theory replaced with cohomology.  This hom group is identified with $H^*_{G,\C,c}(Y,X)$ in \cite[Section 1.E]{Baum:2002aa}.  There does not seem to be a reasonable analogue of this process for non-compact groups $G$, so it is not suitable for our purposes.

Second, Voigt \cite[Theorem 6.6]{Voigt:2009aa} establishes an analog of Theorem \ref{raven chern} where $G$ is allowed to be any second countable, totally disconnected, locally compact group, and $X$ is a locally finite and finite-dimensional $G$-simplicial complex.  Voigt's argument is quite different to Raven's: Voigt starts with the Chern character from $KK^G$ to bivariant equivariant local cyclic homology $HL^G$ that he constructs in \cite{Voigt:207aa}, and then proceeds to directly compute (using other earlier results  \cite{Voigt:2008aa}) that under his assumptions $HL^G$ agrees with the bivariant equivariant homology theory constructed by Baum and Schneider from Definition \ref{bs hom}.  For us, we need the second variable $X$ in Theorem \ref{raven chern} to be allowed to be more general than a $G$-simplicial complex.  Voigt informed us that his Chern character should extend to more general spaces in the second variable: those for which an appropriate 'smooth' subalgebra of the continuous functions exists, such as totally disconnected spaces or manifolds.  However, the details of this are not recorded in the literature, so we use Raven's version here.

Third, L\"{u}ck \cite{Luck:2002aa} constructs a very general equivariant Chern character using methods from algebraic topology (in particular, Bredon homology theories).  L\"{u}ck shows that his Chern character is an isomorphism for many proper equivariant homology theories under natural hypotheses; these apply in particular to equivariant $K$-homology of proper spaces.  The published work on this is only explicitly for single variable theories; however, L\"{u}ck pointed out to us that extending his single variable Chern character to the bivariant Chern case is possible, using naturality in the associated homology theory.  A detailed discussion of the bivariant homology theory that would form the domain for such a bivariant Chern character and also its identification with $KK$-theory under appropriate assumptions were carried out by Mitchener \cite{Mitchener:2004aa}\footnote{Kranz points our that there might be inconsistencies in part of Mitchener's work: see \cite[page 510]{Kranz:2021aa}.} and Kranz \cite{Kranz:2021aa}.  However, an explicit discussion of the bivariant Chern character itself in this language is again missing from the literature.

\end{remark}

If $Z$ has a cofinal family of $G$-compact and second countable $G$-invariant subspaces $Y$, and if $X$ is second countable, then the representable $KK$-group is defined by
$$
RKK^G_*(C_0(Z),C_0(X)):=\lim_{Y\subseteq Z}KK^G_*(C_0(Y),C_0(X)),
$$
where the limit is now restricted to second countable $Y$ (compare \cite[Definition 2.22]{Kasparov:1988dw} and \cite[Definition 5.1 and following paragraph]{Kasparov:2003cf}).  We get the following variant on Theorem \ref{raven chern}.

\begin{corollary}\label{lim cor}
Let $G$ be a countable discrete group, let $\underline{E}G$ be its classifying space for proper actions realized as a $G$-CW complex (see for example \cite[page 6]{Mislin:2003lr}), let $X$ be a second countable, locally compact, Hausdorff $G$-space, and let $H_{G,\C,!}^{*}(\underline{E}G,X)$ be as in Definition \ref{bs!}.  Then there is a Chern character isomorphism
$$
ch_R:RKK^G_*(C_0(\underline{E}G),C_0(X))\otimes \C\to H_{G,\C,!}^{**}(\underline{E}G,X).
$$
\end{corollary}

\begin{proof}
Note that $\underline{E}G$ is the increasing union of its $G$-finite, $G$-invariant subcomplexes; moreover, each such subcomplex is $G$-compact and second countable, and any $G$-compact $G$-invariant subspace of $\underline{E}G$ is eventually contained in such a subcomplex by definition of the CW topology.  Hence by definition (see Definition \ref{bs!} above) for the first equality and Lemma \ref{bs=r} for the second
$$
H_{G,\C,!}^*(\underline{E}G,X)=\lim_{Y\subseteq Z}H^*_{G,\C,c}(Y,X)\cong \lim_{Y\subseteq Z}\widehat{HH}^*_{G,k}(Y,X)
$$
where the limits are both taken over all $G$-finite $G$-invariant subcomplexes $Y$ of $EG$.  Moreover,  
$$
RKK^G_*(C_0(\underline{E}G),C_0(X))=\lim_{Y\subseteq Z}KK^G_*(C_0(Y),C_0(X)),
$$
The result now follows from the existence and naturality of the isomorphisms in Theorem \ref{raven chern}.  
\end{proof}

\section{Main results}

The next theorem and corollary are the main results of the paper. 
\begin{theorem}\label{main}
Let $G$ be a countable discrete group, let $\underline{E}G$ be its classifying space for proper actions realized as a $G$-CW complex (see for example \cite[page 6]{Mislin:2003lr}).  Let $X$ be a locally compact, Hausdorff $G$-space with finite $c$-$k$-cohomological dimension, and let $\widehat{X}$ be as in Definition \ref{z hat}.  Let 
$$
0\to \C_{\widehat{X}}\to I^0\to I^1\to \cdots \to I^m\to 0
$$
be a finite-length resolution of $\C_{\widehat{X}}$ by $c$-soft $G$-sheaves of $\C$-vector spaces, and let $\Gamma_c(I^{-\bullet})$ denote the induced chain complex 
$$
\Gamma_c(\widehat{X},I^m)\leftarrow \cdots \leftarrow \Gamma_c(\widehat{X},I^0)
$$
of $G$-$\C$ modules, where $\Gamma_c(\widehat{X},I^m)$ appears in degree $-m$ and $\Gamma_c(\widehat{X},I^0)$ in degree zero.   Then there is a canonical isomorphism
$$
RKK^G_*(C_0(\underline{E}G),C_0(X)) \otimes \C \cong H_{**}(G,\Gamma_c(I^{-\bullet})).
$$
\end{theorem}

\begin{proof}
We have isomorphisms
$$
RKK^G_*(C_0(\underline{E}G),C_0(X)) \otimes \C\cong H_{G,\C,!}^{**}(\underline{E}G,X) \cong H_{-**}(G,\Gamma_c(I^{-\bullet}))
$$
by Corollary \ref{lim cor} and Proposition \ref{bcr and gh} respectively.  As we are dealing with $\Z/2$-graded homology theories, the minus sign on degrees is irrelevant.
\end{proof}

\begin{corollary}\label{main cor}
Assume the set up of the previous theorem and furthermore suppose that $G$ satisfies the rational Baum--Connes conjecture with coefficients in $C_0(X)$. Then with notation as in Theorem \ref{main}, 
$$
K_*(C_0(X)\rtimes_rG) \otimes \C \cong H_{**}(G,\Gamma_c(I^{-\bullet})).
$$
In particular, if $X$ is totally disconnected and the $G$-action on $X$ has torsion-free stabilizers, then 
$$
K_*(C(X)\rtimes_rG) \otimes \C \cong H_{**}(G,\C[X])\cong H_{**}(G\ltimes X,\C)
$$
and so the rational HK-conjecture holds for $G\ltimes X$.
\end{corollary}

\begin{proof}
By the rational Baum--Connes conjecture, 
\[
K_*(C(X)\rtimes_rG) \otimes \C \cong RKK^G_*(\underline{E}G,C_0(X)) \otimes \C
\]
and by Theorem \ref{main}
$$
RKK^G_*(\underline{E}G,C_0(X)) \otimes \C \cong H_{**}(G,\Gamma_c(I^{-\bullet})),
$$
which completes the proof of the first part. 

The second part follows from the first part, the fact that $\hat{X}=X$ when the $G$ action has torsion-free stabilizers, and the isomorphisms
$$
H_{**}(G,\Gamma_c(I^{-\bullet}))\cong H_{**}(G,\C[X])\cong H_{**}(G\ltimes X,\C)
$$
(compare the comments in Remark \ref{dim rem} on zero-dimensional spaces for the first of these, and the discussion just before Proposition 2.4 in \cite{Matui2016aa} for the second).
\end{proof}
\begin{remark}
It is worth comparing the hypotheses of the rational HK-conjecture with those in Corollary \ref{main cor}. Firstly, unlike in the HK-conjecture, there is no minimality assumption in the corollary. Secondly, in the original version of the rational HK-conjecture the groupoid is assumed to be essentially principal. However, as mentioned in the introduction, Scarparo's counterexample \cite{Scarparo:2020aa} implies that one needs a different hypothesis on the isotropy. A natural one is to assume that the groupoid is principal; in fact, the assumption in Corollary \ref{main cor} is weaker than principal.  Thirdly, the original version of the HK conjecture requires the groupoid to be ample, but here ampleness (i.e.\ zero-dimensionality of the base space) is replaced by a more general finite-dimensionality assumption, at the price of replacing group(oid) homology with hyperhomology.  Due to the failure of the classical Chern character to be an integral isomorphism for higher-dimensional spaces, zero-dimensionality is a natural assumption for the original (integral) HK conjecture.  Finally, Corollary \ref{main cor} only applies to (certain) transformation groupoids while the rational HK-conjecture is stated for (certain) general groupoids. 
\end{remark}

\begin{remark}
For readers familiar with the Baum--Connes conjecture it is worth noting that we do not need to assume the strong Baum--Connes conjecture (by which we mean that the statement that the $\gamma$ element exists and equals one) for the previous result, `just' the statement that the assembly map for $G$ with coefficients in $C_0(X)$ is an isomorphism.  This is a genuinely weaker statement: for example, for a property (T) hyperbolic group, the $\gamma$ element is not one, but Lafforgue \cite{Lafforgue:2009ss} has established the Baum--Connes conjecture with arbitrary coefficients for these groups.

Let us also comment on how to remove the strong Baum-Connes assumption from another approach to the HK conjecture.  Indeed, the Proietti-Yamashita spectral sequence of \cite[Corollary 3.6]{ProiettiYasashita2023aa} gives a different and very interesting `Baum--Connes' based approach to the HK conjecture (compare for example the discussion in \cite[Section 4.4]{BonickeDellAieraGabeWillett2023aa}).  In the statement of \cite[Corollary 3.6]{ProiettiYasashita2023aa}, the authors assume the strong Baum--Connes conjecture.   However, strong Baum--Connes is only necessary as at the time \cite{ProiettiYasashita2023aa} was written, there was not a model for the groupoid Baum--Connes conjecture in terms of localizations of categories as in the work of Meyer-Nest \cite{Meyer:2006fr} for groups.  B\"{o}nicke and Proietti have since provided such a model: one can use \cite[Theorem 3.14]{Bonicke:2022vk} to show that the spectral sequence appearing in \cite[Corollary 3.6]{ProiettiYasashita2023aa} always converges to the left hand side of the groupoid Baum--Connes conjecture in the case of torsion-free isotropy, and thus that strong Baum--Connes is not needed for that result, `just' the statement that the usual Baum--Connes assembly map is an isomorphism.
\end{remark}

\subsection{An HK conjecture that allows torsion stabilizers}

Based on the Theorem \ref{main}, it is natural to introduce a revised version of the rational HK-conjecture. Some notation is required to do so. Let $\mathcal{G}$ be an ample groupoid with base space $\mathcal{G}^{(0)}=X$.  Let $Iso(\mathcal{G})$ be the isotropy subgroupoid of $\mathcal{G}$. $Iso(\mathcal{G})$ is a closed subgroupoid of $\mathcal{G}$, but it need not be \'{e}tale.  An element $g\in Iso(\mathcal{G})$ is \emph{torsion} if $g^n$ belongs to $X$ for some $n\in \N$ with $n>0$.  We define $\widehat{X}$ to consist of the torsion elements of $Iso(\mathcal{G})$ of $\mathcal{G}$ equipped with the subspace topology; note that if $\mathcal{G}$ were a transformation groupoid, this is the same space as was defined in Definition \ref{z hat}.  We define an action of $\mathcal{G}$ on $\widehat{X}$  (recall that groupoid actions were discussed in Section \ref{secPrelim}, see Definition \ref{def:GroupoidAction}) with anchor map 
$$
p:\widehat{X}\to X,\quad x\mapsto r(x)
$$
and action
$$
g\cdot x:=gxg^{-1}.
$$
We write $\widehat{\mathcal{G}}:=\mathcal{G}\ltimes \widehat{X}$ for the associated crossed product, i.e.\
$$
\widehat{\mathcal{G}}:=\{(x,g,y)\in \widehat{X}\times \mathcal{G}\times \widehat{X}\mid s(g)=p(y),~g\cdot y=x\}
$$
equipped with the subspace topology it inherits from $\widehat{X}\times \mathcal{G}\times \widehat{X}$, source and range map given by projection on the third and first factors respectively, multiplication given by $(x,g,y)(y,h,z)=(x,gh,z)$, and inverse given by $(x,g,y)^{-1}=(y,g^{-1},x)$.

\begin{lemma}
Using the notation of the previous paragraph, $\widehat{X}$ is locally compact for the subspace topology it inherits from $\mathcal{G}$ and $\widehat{\mathcal{G}}$ is a locally compact, Hausdorff, \'{e}tale groupoid.
\end{lemma}

\begin{proof}
We first claim that $\widehat{X}$ is open in $Iso(\mathcal{G})$.  Indeed, consider the map 
$$
p_n:Iso(\mathcal{G})\to \mathcal{G},\quad g\mapsto g^n.
$$
Then $p_n$ is continuous, and $\widehat{X}=\bigcup_{n=1}^\infty p_n^{-1}(X)$.  As $\mathcal{G}$ is \'{e}tale, $X$ is open in $\mathcal{G}$, whence $\widehat{X}$ is open in $Iso(\mathcal{G})$, giving the claim.  In particular, the claim implies that $\widehat{X}$ is an open subset of a closed subset of a locally compact space, so locally compact in its own right.

We now look at $\widehat{\mathcal{G}}$.  This is a closed subspace of $\widehat{X}\times \mathcal{G}\times \widehat{X}$, whence locally compact and Hausdorff, and the operations are all continuous by properties of product topologies, and continuity of the operations on $\mathcal{G}$.  The range and source maps are open as they are restrictions of coordinate projections.  It remains to show that the range map is a local homeomorphism; given we already know that it is continuous and open, it suffices to show that it is locally injective.  Indeed, let $(x,g,y)\in \widehat{\mathcal{G}}$ be given, and let $U\owns g$ and $V\owns y$ be open sets on which the range map for $\mathcal{G}$ is injective.  We claim that the range map for $\widehat{\mathcal{G}}$ is injective on $\widehat{X}\times U\times V \cap \widehat{\mathcal{G}}$.  Indeed, if $(x_1,g_1,y_1)$ and $(x_2,g_2,y_2)$ are points in this set such that $x_1=x_2$, then $s(x_1)=s(x_2)$, so $r(g_1)=r(g_2)$, so $g_1=g_2$ by choice of $U$; hence also $s(g_1)=s(g_2)$, so $r(y_1)=r(y_2)$, so $y_1=y_2$ by choice of $V$ and we are done.
\end{proof}

Now, recall that Crainic and Moerdijk \cite[3.3]{CrainicMoerdijk2000aa} show that if $\mathcal{G}$ is a groupoid with base space $\mathcal{G}^{(0)}$ and $c$-$\C$-$\text{dim}(\mathcal{G}^{(0)})<\infty$, then there is a finite length resolution 
$$
0\to \C_{\mathcal{G}^{(0)}}\to I^0\to \cdots \to I^m\to 0
$$
of the sheaf $\C_{\mathcal{G}^{(0)}}$ of locally constant $\C$-valued functions on $\C_{\mathcal{G}^{(0)}}$ by $c$-soft $\mathcal{G}$-sheaves of $\C$-vector spaces (the proof is essentially the same as that of Lemma \ref{rgcx} below).  Moreover, in \cite[3.1]{CrainicMoerdijk2000aa}, Crainic and Moerdijk define the groupoid hyperhomology groups $\mathbb{H}_*(\mathcal{G};I^{-\bullet})$ associated to such a resolution, and in \cite[3.4]{CrainicMoerdijk2000aa}, they define the groupoid homology to be $H_*(\mathcal{G};\C_{\widehat{X}}):=\mathbb{H}_*(\mathcal{G};I^{-\bullet})$.

By analogy with Theorem \ref{main} above, one is led to the following general version of the HK conjecture.

\begin{conjecture}\label{hk revised}
Suppose that $\mathcal{G}$ is a second countable, \'etale groupoid such that $c$-$\C$-$\text{dim}(\widehat{X})<\infty$ and such that the Baum--Connes conjecture holds for $\mathcal{G}$. Then 
\begin{equation}\label{new hk}
K_*(C^*_r(\mathcal{G}))\otimes \C\cong H_{**}(\widehat{\mathcal{G}};\C_{\widehat{X}}).
\end{equation}
\end{conjecture}

Corollary \ref{main cor} establishes this conjecture for transformation groupoids.  The right hand side of line \eqref{new hk} above should really be regarded as a conjectural computation of the left hand side of the Baum--Connes conjecture for $\mathcal{G}$: we also conjecture that the right hand side of line \eqref{new hk} above is always isomorphic to the left hand side of the Baum--Connes conjecture (tensored with $\C$) for the groupoid $\mathcal{G}$; Theorem \ref{main} establishes this conjecture for transformation groupoids.

\subsection{Computational tools and examples}

In this section we discuss some explicit computational tools, use these to discuss how Scarparo's counterexamples to the HK conjecture \cite{Scarparo:2020aa} interact with Theorem \ref{main}, and discuss two additional examples.

We first give a somewhat more explicit (although unnatural in $G$) computation of $H_*(G,k[\widehat{X}])$ when $k$ is a unital commutative ring.  This requires us to recall a standard definition: compare for example \cite[Section III.5]{Brow:1982rt}.

\begin{definition}\label{ind def}
Let $H$ is a subgroup of a group $G$ and $M$ is an $H$-$k$-module, then the \emph{induced} module is $\text{Ind}_H^G(M):=kG\otimes_{kH}\otimes M$.
\end{definition}

\begin{proposition} \label{Rem:ComHGamma}
Let $X$ be a totally disconnected, locally compact, Hausdorff $G$-space.  For each conjugacy class $c\in G_{tor}//G$, fix $g_c\in c$.  Then 
$$
H_*(G,k[\widehat{X}])\cong \bigoplus_{c\in G_{tor}//G} H_*(Z(g_c),k[X^{g_c}]).
$$
\end{proposition}

\begin{proof}
Note that we may write $\widehat{X}$ as
$$
\widehat{X}=\bigsqcup_{c\in G_{tor}//G}\Bigg(\bigsqcup_{g\in c} X^g\times \{g\}\Bigg),
$$
where for each $c$, the set $\bigsqcup_{g\in c} X^g\times \{g\}$ is $G$-invariant.  Hence
\begin{equation}\label{big hat split}
k[\widehat{X}]=\bigoplus_{c\in G_{tor}//G}k\Bigg[\bigsqcup_{g\in c} X^g\times \{g\}\Bigg].
\end{equation}
Note that the centralizer $Z(g_c)$ acts on $X^{g_c}$, and that there are natural identifications of $G$-modules 
\begin{equation}\label{is ind}
k\Bigg[\bigsqcup_{g\in c} X^g\times \{g\}\Bigg]=\text{Ind}_{Z(g_c)}^G k[X^{g_c}].
\end{equation}
On the other hand, Shapiro's lemma (see for example \cite[Proposition III.6.2]{Brow:1982rt}) implies that 
\begin{equation}\label{shap lem}
H_*(G,\text{Ind}_{Z(g_c)}^G k[X^{g_c}])\cong H_*(Z(g_c),k[X^{g_c}]).
\end{equation}
Lines \eqref{big hat split}, \eqref{is ind} and \eqref{shap lem} plus that homology commutes with direct sums gives the result.
\end{proof}

As mentioned in the introduction, Scarparo \cite{Scarparo:2020aa} constructed the first counterexamples to the HK conjecture.  Here we give a brief discussion explaining how Scarparo's examples are consistent with our revised version of the rational HK conjecture, see Conjecture \ref{hk revised}. It should be noted that Scarparo's counterexamples fit within the framework of Theorem \ref{main}, so the revised version of the rational HK conjecture holds for these examples.  In the example below, we use Proposition \ref{Rem:ComHGamma} to make this explicit.

\begin{example}
Let $G=\Z\rtimes (\Z/2)$ denote the infinite dihedral group; the semi-direct product is defined by having $\Z/2$ act by the (unique non-trivial) automorphism of $\Z$ defined by $n\mapsto -n$.  Scarparo writes elements of this group as pairs $(n,i)\in \Z\times (\Z/2)$ subject to the usual multiplication rules for a semi-direct product; precisely, 
$$
(n,i)(m,j)=(n+(-1)^im,i+j).
$$
One computes that the torsion elements of $\Gamma$ are precisely the identity, and those of the form $(n,1)$ for some $n\in \Z$, and that there are two conjugacy classes of such elements: those where $n$ is even, and where it is odd.  Moreover, in either case the centralizer of a non-trivial torsion element is just the subgroup it generates.

Now, choose a sequence $(n_k)_{k=1}^\infty$ of positive integers such that $n_k$ divides $n_{k+1}$ for all $k$.  Define $G_k$ to be the subgroup $\{(n,i)\in G\mid n_k\text{ divides n}\}$ of $G$ (one might reasonably also write $G_k=n_k\Z\rtimes (\Z/2)$).  As $n_k$ divides $n_{k+1}$ we have natural inclusions whence natural surjections $G/G_k\to G/G_{k+1}$ of (finite!) $G$-spaces (note that $G_k$ is \emph{not} normal in $G$, so $G/G_k$ is not a quotient group of $G$).   Define now 
 $$
 X:=\lim_{\leftarrow} G/G_k
 $$
 where the inverse limit is taken in the category of $G$-spaces; thus $X$ is topologically a Cantor set, equipped with an action of $G$ by homeomorphisms.  Scarparo shows in  \cite[Example 2.2]{Scarparo:2020aa} that $X$ is topologically free.  We let $\mathcal{G}:=G\ltimes X$ be the corresponding transformation groupoid.

Now, fix such a sequence $(n_k)$ as above, and define $R:=\{m/n_k\in \Q\mid m\in \Z,k\geq 1\}$, a subgroup of the additive group $\Q$.  Let $m$ be the number of $G$-orbits or points in $X$ with non-trivial isotropy; in \cite[Lemma 3.2]{Scarparo:2020aa} Scarparo shows\footnote{\label{cc slop} We are compressing the discussion slightly for simplicity: Scarparo also discusses which of the conjugacy classes of finite subgroups discussed above gives the stabilizers in each case, and there are in fact three cases to consider if one makes that finer distinction.} that $m$ is either $1$ or $2$, and moreover that in either case the stabilizer of a point in one of the orbits is a copy of $\Z/2$ 

 In \cite[Section 3]{Scarparo:2020aa} (see in particular Proposition 3.3 and Theorem 3.5), Scarparo makes the following computations
\begin{equation}\label{sc k}
K_i(C^*_r(\mathcal{G}))\cong \left\{\begin{array}{ll} R \oplus \Z^m & i=0 \\ 0 & i=1\end{array}\right.
\end{equation}
and 
\begin{equation}\label{sc hom}
(H_i(\mathcal{G})=)~H_i(\Gamma,\Z[X])\cong \left\{\begin{array}{ll} R & i=0 \\ 0 & i\geq 1 \text{ even} \\ (\Z/2)^m  & i \text{ odd} \end{array}\right.
\end{equation}
This shows that the HK conjecture fails, even rationally.  

However, in our case, we want to consider $H_i(G,\Z[\widehat{X}])$.  We have that as a $G$-space $\widehat{X}=X\sqcup (G / (\Z/2))^m$ for some copies\footnote{As in footnote \ref{cc slop} above, one should really be careful as to which precise copy of $\Z/2$ inside $G$ that is appearing here. However, this turns out to be irrelevant to our computations.} of $\Z/2$ inside $G$.  Using Proposition \ref{Rem:ComHGamma}, we see that 
\begin{equation}\label{hat split}
H_i(G,\Z[\widehat{X}])=H_i(G,\Z[X])\oplus H_i(\Z/2,\Z)^m.
\end{equation}
  On the other hand, we have the well-known computation  
\begin{equation}\label{z2 hom}
H_i(\Z/2)\cong \left\{\begin{array}{ll} \Z & i=0 \\ 0 & i\geq 1 \text{ even} \\ \Z/2  & i \text{ odd} \end{array}\right.
\end{equation}
(see for example \cite[Section II.3]{Brow:1982rt}).  Combining lines \eqref{hat split} and \eqref{z2 hom} with Scarparo's computation from line \eqref{sc hom} gives that 
$$
H_i(G,\Z[\widehat{X}])\cong \left\{\begin{array}{ll} R\oplus \Z^m & i=0 \\ 0 & i\geq 1 \text{ even} \\ (\Z/2)^{3m}  & i \text{ odd} \end{array}\right.
$$
and comparing this with line \eqref{sc k} shows that we do indeed have a rational isomorphism 
$$
K_*(C^*_r(\mathcal{G}))\otimes \Q\cong H_{**}(G,\Z[\widehat{X}])\otimes \Q
$$
consistently with the discussion in Conjecture \ref{hk revised} above.

Note that for Scarparo's examples above, we actually have a stronger result than a rational isomorphism. There is an isomorphism
$$
\frac{K_*(C^*_r(\mathcal{G}))}{K_*(C^*_r(\mathcal{G}))_{tor}}\cong \frac{H_{**}(G,\Z[\widehat{X}])}{H_{**}(G,\Z[\widehat{X}])_{tor}}
$$
of the respective quotients by the torsion subgroups.  It is tempting to believe this (or something similar - for example that one has an isomorphism after tensoring with a ring $k$ satisfying the assumptions of Proposition \ref{bcr and gh}) is true in more generality.  
\end{example}

\begin{proposition}\label{spec seq}
With the assumptions of Corollary \ref{main cor}, there is a convergent spectral sequence with terms on the $E^2$ page given by
$$
E_{pq}^2=H_p(G,H^{-q}(\widehat{X},\C)),
$$
and so that the $\Z/2$-graded homology theory associated to the limit (which is naturally $\Z$-graded) is isomorphic to $K_*(C(X)\rtimes_r G)\otimes \C$.
\end{proposition}

\begin{proof}
The cohomology of a complex $\Gamma_c(\widehat{X},I^{\bullet})$ as in the statement of Theorem \ref{main} is (by definition, and up to a sign change on degrees) the compactly sheaf cohomology $H_c^{\bullet}(\widehat{X},\C_{\widehat{X}})$ of $X$ with coefficients in the sheaf $\C_{\widehat{X}}$ of locally constant functions; we take this to be the definition of compactly supported cohomology with coefficients in $\C$ (for `reasonable' $X$ it is isomorphic to other standard models of compactly supported cohomology: see for example \cite[Part III]{Bredon:1997aa}).  

There is therefore a convergent hyperhomology spectral sequence
$$
H_p(G,H^{-q}_c(\widehat{X},\C))=H_p(G,H^{-q}(\Gamma_c(I^\bullet)))~\Rightarrow~H_{p+q}(G,\Gamma_c(I^{-\bullet}))
$$
(see for example \cite[Proposition 5.7.6]{Weibel:1995ty}).  Now apply Corollary \ref{main cor}.
\end{proof}

\begin{example}
The spectral sequence from Proposition \ref{spec seq} can have non-trivial differentials.  To see this, let $G$ be the fundamental group of a closed Riemann surface $M$ of genus at least two considered as acting on the hyperbolic plane identified with the universal cover $\widetilde{M}$ of $M$.  Let $X=S^1$ be the boundary at infinity of $\widetilde{M}$ equipped with the induced $G$-action.  Then $G$ is torsion-free so $\widehat{X}=X$ and
$$
H^{-q}_c(X,\C))=\left\{\begin{array}{ll} \C, & q=-1,0 \\ 0,& \text{otherwise}\end{array}\right..
$$
Moreover, $G$ acts on $S^1$ via orientation-preserving homeomorphisms, so trivially on cohomology.  The non-trivial part of the $E^2$ page of the spectral sequence from Proposition \ref{spec seq} therefore looks like
{\tiny $$
\xymatrix{ q=0 & : & H_0(G,\C) &  H_1(G,\C) & H_2(G,\C)  & H_3(G,\C) & \cdots \\
q=-1 & : & H_0(G,\C) &  H_1(G,\C) & H_2(G,\C) \ar[ull] & H_3(G,\C) \ar[ull] &\cdots \ar[ull] }
$$}
Moreover, the differentials on all the higher pages vanish for dimension reasons.  On the other hand, $M$ is a model for the classifying space $BG$ of $G$ in this case, and so if $g$ is the genus of $m$ then we have the well-known computation 
$$
H_p(G,\C)=H_p(M,\C)=\left\{\begin{array}{ll} \C, & p=0,2 \\ \C^{2g}, & p=1 \\ 0, & \text{otherwise}\end{array}\right.
$$
and the non-trivial part of the $E^2$-page spectral sequence becomes
\begin{equation}\label{diff map}
\xymatrix{ q=0 & : & \C &  \C^{2g} & \C   & \\
q=-1 & : & \C &  \C^{2g} & \C \ar[ull] & }.
\end{equation}
Comparing this, the fact that the spectral sequences converges to $K_{p+q}(C(X)\rtimes_r G)\otimes \C$, and the computation 
$$
K_n(C(S^1)\rtimes G)\cong \left\{\begin{array}{ll} \Z/(2g-2) \oplus \Z^{2g+1}, & n \text{ even} \\ \Z^{2g+1}, & n \text{ odd} \end{array}\right.
$$
from \cite[Example 34]{Emerson:2006uq}, we see that the differential appearing in line \eqref{diff map} above is necessarily non-trivial.  We suspect that in this example (and in other related boundary actions as studied in \cite{Emerson:2006uq}) the differentials in the spectral sequence should be related to the classical Gysin map (compare for example \cite[pages 177-179]{Bott:1982aa}), but we did not seriously pursue this.
\end{example}

\begin{example}
In \cite{Deeley2023aa}, the transformation groupoids obtained from generalized odometer actions of the fundamental group of flat manifolds are considered. In particular, Theorem 3.5 of \cite{Deeley2023aa} establishes the rational HK-conjecture for these groupoids. We will not go into a detailed discussion of this class of groupoids, but only mention that the fact that the rational HK-conjecture holds for them now follows from Corollary \ref{main cor}. The reason Corollary \ref{main cor} can be applied is because the fundamental group of a flat manifold is torsion-free and amenable.  
\end{example}

\appendix

\section{Baum-Schneider homology and group hyperhomology}\label{hh app}

In this appendix, we give a proof of of Proposition \ref{bcr and gh}, which we restate below for the reader's convenience.

\begin{proposition*}
Let $X$ be a locally compact, Hausdorff $G$-space with finite $c$-$k$-cohomological dimension, and let $\widehat{X}$ be as in Definition \ref{z hat}.  Let 
$$
0\to k_{\widehat{X}}\to I^0\to I^1\to \cdots \to I^m\to 0
$$
be a finite-length resolution of $k_{\widehat{X}}$ by $c$-soft $G$-sheaves\footnote{Such a resolution always exists: see Remark \ref{dim rem} above.} of $k$-modules, and let $\Gamma_c(I^{-\bullet})$ denote the induced chain complex 
$$
\Gamma_c(\widehat{X},I^m)\leftarrow \cdots \leftarrow \Gamma_c(\widehat{X},I^0)
$$
of $G$-modules, where $\Gamma_c(\widehat{X},I^m)$ appears in degree $-m$ and $\Gamma_c(\widehat{X},I^0)$ in degree zero.  Let $k$ be a Pr\"{u}fer domain\footnote{See Remark \ref{k rem} above for more about what this assumptions means, and examples.} such that $n^{-1}$ exists in $k$ whenever there is an order $n$ stabilizer of a point $x\in X$.  Let $\underline{E}G$ be the universal $G$-space for proper actions.  

Then for each $n\in \Z$, there is a canonical isomorphism
$$
H^n_{G,k,!}(\underline{E}G,X)\cong H_{-n}(G,\Gamma_c(I^{-\bullet}))
$$ 
between the Baum-Schneider bivariant cohomology groups and the group hyperhomology with coefficients in the $G$-$k$ chain complex $\Gamma_c(I^{-\bullet})$.
\end{proposition*}

The proof of Proposition \ref{bcr and gh} is long; we summarize the main steps here.
\begin{enumerate}
\item[\textbf{Step 1}] (See subsection \ref{s1} below).  Compute $R\Gamma_c k_{\widehat{Y}}$ when $Y$ is a proper, $G$-finite $G$-simplicial complex satisfying an appropriate orientation assumption.  The key point, which is important for Step 3, is that $R\Gamma_c k_{\widehat{Y}}$ can be computed by a complex of projective $G$-$k[G_{tor}]$ modules.
\item[\textbf{Step 2}] (See subsection \ref{s2} below). Compute $R\Gamma_ck_{\widehat{X}}$ when $X$ has finite $c$-$k$-cohomological dimension.  The key point, important for Step 5, is that this can be computed by modules that have good flatness properties (they are not flat in the category of $G$-$k[G_{tor}]$ modules, however).  This uses standard machinery from sheaf theory.
\item[\textbf{Step 3}] (See subsection \ref{s3} below).  Compute $H^*_{G,k,c}(Y,X)=\text{Ext}^*_{G\text{-}k[G_{tor}]}(R\Gamma_c k_{\widehat{Y}},R\Gamma_c k_{\widehat{X}})$.  Having done the computations in Steps 1 and 2, this is a direct algebraic computation.
\item[\textbf{Step 4}] (See subsection \ref{s4} below).  Compute $H^*_{G,k,!}(\underline{E}G,X)$ as the cohomology of an explicit double complex. This is done by writing $\underline{E}G$ as an increasing union of simplicial complexes satisfying the assumptions of Step 1; this is the shortest step.
\item[\textbf{Step 5}] (See subsection \ref{s5} below).  Compute $H_*(G,\Gamma_c(I^{-\bullet}))$.  The point is to show that this can be computed from the same double complex that we found in Step 4.  
\end{enumerate}

The most technical steps are 1 and 5; both adapt fairly standard ideas from algebraic topology.  The argument is based on the material from \cite[Section 1.B]{Baum:2002aa} and \cite[Section 4]{Schneider:1998aa}.  However, we need some refinements based on the fact that we are working with more general rings than $k=\C$; we also provide more details and complete references than in these sources for the benefit of non-expert readers.

\subsection{Step 1: computation of $R\Gamma_ck_{\widehat{Y}}$ for $Y$ an appropriate simplicial complex}\label{s1}

We need some more definitions.  The following conventions are based on \cite[page 317]{Baum:2002aa}.  

\begin{definition}\label{gsc}
A \emph{$G$-simplicial complex} $Y$ consists of a set $Y_0$ and for each $i\geq 1$ a collection $Y_i$ of $(i+1)$-element subsets of $Y_0$ with the following properties:
\begin{enumerate}[(i)]
\item if $\sigma\in Y_i$ and $\eta\subseteq \sigma$ has $j+1$ elements, then $\eta\in Y_j$;
\item $G$ acts on $Y_0$ in such a way that the induced action on the power set of $Y_0$ preserves each $Y_i$.
\end{enumerate}
The elements of $Y_i$ are called \emph{$i$-simplices} and the elements of $Y_0$ are (also) called vertices.  If $Y_d\neq \varnothing$ and $Y_i=\varnothing$ for $i>d$, then $d$ is called the \emph{dimension} of $Y$; if such a $d$ exists, then $Y$ is \emph{finite dimensional}.

The \emph{geometric realization} of $Y$, denoted $|Y|$, is the subset of $[0,1]^{Y_0}$ consisting of all tuples $(t_{\sigma})_{\sigma\in Y_0}$ such that $\{\sigma\in Y_0\mid t_\sigma\neq 0\}$ is an element of $Y_i$ for some $i$, and such that $\sum_{\sigma\in Y_0}t_\sigma=1$.  It is equipped with the induced $G$-action from the $G$-action on $[0,1]^{Y_0}$.  For each finite subset $S$ of $Y_0$, the set $|Y|_S:=\{(t_{\sigma})_{\sigma\in Y_0}\in |Y|\mid t_\sigma=0 \text{ for } \sigma\not\in S\}$ is given the subspace topology it inherits from $[0,1]^S$.  The space $|Y|$ is then equipped with the direct limit topology it inherits by writing it as a union $|Y|=\bigcup_{S\subseteq Y_0 \text{ finite}}|Y|_S$, i.e.\ a subset $U$ of $|Y|$ is open if and only if $U\cap |Y|_S$ is open for all finite $S\subseteq Y_0$.

A $G$-simplicial complex $Y$ is:
\begin{enumerate}[(a)]
\item \emph{proper} if the induced action on $|Y|$ is proper (in particular, this implies that for each $i$ and each $\sigma\in Y_i$, the stabilizer $G_\sigma$ is finite);
\item \emph{$G$-finite} if it is finite dimensional and such that $Y_i/G$ is finite or all $i$;
\item \emph{type-preserving} if for any simplex $\sigma=\{\sigma_0,...,\sigma_i\}\in Y_i$ (with each $\sigma_i\in Y_0$) the induced action of $G_\sigma$ on $\{\sigma_0,...,\sigma_i\}$ is trivial\footnote{In other words, if an element of $G$ fixes a simplex, then it fixes all vertices of that simplex.};
\item \emph{$G$-oriented} if there is a fixed $G$-invariant partial order on $Y_0$ that restricts to a total order on each simplex\footnote{As intimated, for example, on \cite[page 107]{Hatcher:2002ud}, this structure enables one to carry out simplicial-type homology computations; moreover, as the structure is $G$-invariant, we will be able to do so equivariantly}.
\end{enumerate}
\end{definition}

\begin{remarks}\label{gsc rem}
\begin{enumerate}[(i)]
\item If $Y$ is $G$-finite and if all simplex stabilizers $G_\sigma$ are finite, then $|Y|$ is a locally compact, proper, Hausdorff $G$-space.
\item \label{g-or rem} If $Y$ is a $G$-simplicial complex, then a $G$-orientation of $Y$ exists if and only if each $G$-orbit in $Y_0$ intersects any simplex at most once.  Indeed, if $Y$ has a $G$-orientation $<$, then by $G$-invariance, any distinct points of $Y_0$ in the same $G$-orbit cannot be related by $<$; as $<$ restricts to a total order on each simplex, each $G$-orbit can intersect each simplex at most once.   Conversely, assume every $G$-orbit in $Y_0$ intersects each simplex at most once.  Choose any total ordering $<_G$ on $Y_0/G$, and define a partial order on $Y_0$ by stipulating $\sigma<\eta$ if and only if $[\sigma]<_G [\eta]$.  As no two points of any simplex are in the same orbit, this restricts to a total order on each simplex.

In particular, it follows from the discussion above that if $Y$ admits a $G$-orientation, then the action is necessarily type-preserving.  Nonetheless, we typically state both assumptions separately in the results that follow as this seems more explicit.
\item \label{bary rem} If $Y$ is a $G$-simplicial complex, its \emph{barycentric subdivision} $Y^{(b)}$ is the simplicial complex with vertex set consisting of all simplices $Y$, and where a collection $\{\sigma_0,...,\sigma_i\}$ form an $i$-simplex if there are proper\footnote{We use the symbol ``$\subsetneq$'' for a proper inclusion.} inclusions $\sigma_0\subsetneq \sigma_1\subsetneq \cdots\subsetneq  \sigma_i$.  The $G$-action on $Y$ naturally induces an action on $Y^{(b)}$.  There is moreover a canonical ordering on $Y^{(b)}_0$ given by stating that $\eta\leq \sigma$ if $\eta\subseteq \sigma$; hence $Y^{(b)}$ is canonically $G$-oriented (and type-preserving).  

As $Y$ and $Y^{(b)}$ have the same geometric realization, this suggests in particular that the type-preserving and $G$-orientedness assumptions are not too onerous.
\end{enumerate}
\end{remarks}

We are now ready to define one of the key objects needed for our computation of $R\Gamma_ck_{\widehat{Y}}$.

\begin{definition}\label{simp der def}
Let $Y$ be a $G$-finite, proper, type-preserving and $G$-oriented $G$-simplicial complex of dimension $d$.  For each $i\in \{0,...,d\}$, define $M_i:=\bigoplus_{\sigma\in Y_i}k[G_\sigma]$.  We make each $M_i$ a $G$-$k[G_{tor}]$-module by letting $k[G_{tor}]$ act by pointwise multiplication on each summand, and via the $G$-action
\begin{equation}\label{g act}
g\cdot (a_\sigma)_{\sigma\in Y_i}=(ga_{g^{-1}\sigma}g^{-1})_{\sigma\in Y_i}
\end{equation}
(this makes sense as if $h\in G_{g^{-1}\sigma}$ then $ghg^{-1}$ is in $G_\sigma$).

For each $i$, each $\eta\in Y_i$ and for each $j\in \{0,...,i\}$, let $\eta^{(j)}$ denote the $j^\text{th}$ face of $\eta$; i.e.\ if $\eta=\{\eta_0,...,\eta_i\}$ (with the order induced by the $G$-orientation), $\eta^{(j)}$ has the same vertices as $\eta$ but with the $j^\text{th}$ vertex removed.  Define a map $\partial_j:k[G_\sigma]\to M_{i+1}$ by stipulating that the component of $\partial_j(a)$ in the summand $k[G_\eta]$ is $a|_{G_\eta}$ if $\eta^{(j)}=\sigma$, and $0$ otherwise; note that if $\eta^{(j)}=\sigma$ then $G_\eta$ is a subgroup of $G_\sigma$ by our type preserving assumption so the restriction $a|_{G_\eta}$ makes sense; note also that $\partial_j(a)$ can only have finitely many non-zero components as the assumptions that $Y$ is proper and $G$-finite imply that $\sigma$ can only appear as a face of finitely many higher-dimensional simplices.   Define $\partial_j:M_i\to M_{j+1}$ by using the previous definition on each summand $k[G_\sigma]$, which is a map of $G$-$k[G_{tor}]$-modules.  Define 
$$
\partial:M_i\to M_{i+1},\quad \partial:=\sum_{j=0}^n (-1)^j \partial_j.
$$
The resulting sequence 
$$
\bigoplus_{\sigma\in Y_0}k[G_\sigma]\stackrel{\partial}{\to} \cdots \stackrel{\partial}{\to} \bigoplus_{\sigma\in Y_d} k[G_\sigma]
$$
of abelian $G$-$k[G_{tor}]$ modules is then a cochain complex\footnote{It might help situate some readers to point out that the maps $\partial_j$ make the sequence $(M_i)_{i=0}^d$ into a \emph{semi-simplicial object} in $G$-$k[G_{tor}]$-mod in the sense of \cite[Definition 8.1.9]{Weibel:1995ty}, and the chain complex above is the usual (unnormalized) associated complex as in \cite[Definition 8.2.1]{Weibel:1995ty}.}, which we call the \emph{basic complex} of $Y$.
\end{definition}

If $Y$ is a $G$-simplicial complex, we abuse notation slightly by writing $\widehat{Y}$ for the corresponding $G$-space from Definition \ref{z hat} that should more properly be called $\widehat{|Y|}$.  

We need to recall some definitions and basic facts from sheaf theory.

\begin{definition}\label{g r}
Let $F$ be a sheaf over a topological space $Z$, and write $F_z$ for the stalk of $F$ over $z\in Z$.  The \emph{sheaf of discontinuous sections of $F$} is the sheaf $F_{disc}$ whose sections $F_{disc}(U)$ over an open set $U\subseteq Z$ are given by $F_{disc}(U):=\prod_{z\in U}F_z$.  Note that we have a canonical embedding of $F$ into $F_{disc}$: over an open set $U$, this is defined by sending a section $s\in F(U)$ to the section $(s_z)_{z\in U}\in F_{disc}(U)$ consisting of all its germs over $U$.  Note that if $F$ is a $G$-sheaf, then $F_{disc}$ is also a $G$-sheaf.

The \emph{Godement resolution} $0\to F\to F^0\to F^1\to \cdots$ of $F$ is defined by first taking $F^0=F_{disc}$ equipped with the canonical embedding of $F$ in $F_{disc}$, then taking the quotient sheaf $F^0/F$, embedding that in the corresponding sheaf $F^1:=(F/F^0)_{disc}$ of possibly discontinuous sections, and so on.
\end{definition}

Compare \cite[Definition 1.8.2]{Kashiwara:1990aa} for the following definition.

\begin{definition}\label{f-inj}
    Let $K:\mathcal{A}\to \mathcal{B}$ be an additive functor between abelian categories.  A full additive subcategory $\mathcal{I}$ of $\mathcal{A}$ is \emph{injective with respect to $K$} if the following hold:
\begin{enumerate}[(i)]
\item any object of $\mathcal{A}$ embeds\footnote{In a general abelian category, an \emph{embedding} is a morphism with zero kernel.} in an object from $\mathcal{I}$; 
\item if $i:I\to J$ is an embedding in $\mathcal{I}$ then the cokernel of $i$ is in $\mathcal{I}$; 
\item the restriction of $K$ to $\mathcal{I}$ takes short exact sequences to short exact sequences.  
\end{enumerate}
\end{definition}

We now put these definitions to use in an auxiliary lemma.  Most of this is implicit in \cite[Proposition 2]{Schneider:1998aa}.

\begin{lemma}\label{g c soft}
Let $Z$ be a locally compact, Hausdorff $G$-space.  
\begin{enumerate}[(i)]
\item \label{gcs god} If $F$ is a $G$-sheaf sheaf over $Z$ (or over $\widehat{Z}$), then the Godement resolution from Definition \ref{g r} consists of $c$-soft $G$-sheaves in the sense of Definition \ref{c-soft}.
\item \label{gcs inj} The full subcategory of $\text{Sh}_G(\widehat{Z})$ of $c$-soft $G$-sheaves of $k$-modules is injective (in the sense of Definition \ref{f-inj}) with respect to the functor $\Gamma_c:\text{Sh}_G(\widehat{Z})\to G\text{-}k[G_{tor}]$ from Definition \ref{gammac}.  
\item \label{gcs r} If $F$ is a $G$-sheaf over $\widehat{Z}$ and if  
$$0\to F\to F^0\to F^1\to \cdots$$ 
is a resolution in $\text{Sh}_G(\widehat{Z})$ consisting of $c$-soft $G$-sheaves, then $R\Gamma_cF$ is isomorphic in $\dcc(G\text{-}k[G_{tor}])$ to the complex 
$$
\Gamma_c(F^0)\to\Gamma_c(F^1)\to \cdots.
$$
\end{enumerate}
\end{lemma}

\begin{proof}
For point \eqref{gcs god}, it suffices to show that if $F$ is a $G$-sheaf, then the sheaf $F_{disc}$ inherits a $G$-action from $F$, and is $c$-soft; both of these points are immediate.

For point \eqref{gcs inj}, note first that any $G$-sheaf $F$ embeds in $F_{disc}$, which is a $c$-soft $G$-sheaf by \eqref{gcs god}.  On the other hand, the subcategory of the category $\text{Sh}(\widehat{Z})$ of non-equivariant sheaves consisting of $c$-soft sheaves is injective with respect to the functor $\Gamma_c:\text{Sh}(\widehat{Z})\to k\text{-mod}$ by the discussion on \cite[Page 105]{Kashiwara:1990aa}; as a sequence in $\text{Sh}_G(\widehat{Z})$ (respectively, in $G\text{-}k[G_{tor}]\text{-mod}$) is exact if and only if its image under the forgetful functor to $\text{Sh}(\widehat{Z})$ (respectively, to $k\text{-}\text{mod}$) is exact, this gives the $\Gamma_c$-injectivity result.

Point \eqref{gcs r} follows from the general fact that given an additive functor $K:\mathcal{A}\to \mathcal{B}$ between abelian categories and object $A$ of $\mathcal{A}$, the image of $A$ under the derived functor $RK:\dcc(\mathcal{A})\to \dcc(\mathcal{B})$ can be computed by taking the image of a $K$-injective resolution of $A$ under $K$: see \cite[Proposition 1.8.3]{Kashiwara:1990aa}.  
\end{proof}

We need one more piece of notation, which is based on \cite[I.2.6]{Bredon:1997aa} (see also \cite[page 93]{Kashiwara:1990aa} for an equivalent definition in the language of presheaves).

\begin{definition}\label{f sub z}
Let $Z$ be a topological space, and let $Y\subseteq X$ denote a locally closed subspace.  Let $F$ denote a sheaf of $k$-modules over $Z$ with associated \'{e}tale space $\pi:\mathcal{F}\to Z$.  We denote by $F_Y$ the sheaf whose \'{e}tale space is 
$$
\{f\in \mathcal{F}\mid \pi(f)\in Y \text{ or } f=0\}
$$
equipped with the unique topology for which $\pi^{-1}(Y)$ is embedded, and for which the restriction of $\pi$ is a local homeomorphism.
\end{definition}

Note that if $Z$ is a $G$-space and $Y$ is a locally closed $G$-invariant subspace, then $F_Y$ is canonically a $G$-sheaf for the restricted action.

\begin{lemma}\label{g-fin}
Let $Y$ be a proper, $G$-finite, type-preserving, $G$-oriented $G$-simplicial complex of dimension $d$.  Let $k$ be a commutative unital ring and let $k_{\widehat{Y}}$ denote the $G$-sheaf of locally constant $k$-valued functions on the space $\widehat{Y}$ of Definition \ref{z hat}.   Then the image $R\Gamma_ck_{\widehat{Y}}$ of $k_{\widehat{Y}}$ under the total derived functor $R\Gamma_c$ of Definition \eqref{r der} is isomorphic in $\dcc(G\text{-}k[G_{tor}])$ to the basic complex 
$$
\bigoplus_{\sigma\in Y_0}k[G_\sigma]\stackrel{\partial}{\to} \cdots \stackrel{\partial}{\to} \bigoplus_{\sigma\in Y_d} k[G_\sigma]
$$
of Definition \ref{simp der def}.
\end{lemma}

The reader might usefully compare what follows to the standard proof that the singular and cellular cohomology of a CW complex agree: see for example \cite[pages 137-141]{Hatcher:2002ud}

\begin{proof}
For each $i\in \{0,...,d\}$, let $\widehat{Y}_i\subseteq \widehat{Y}$ be the pullback to $\widehat{Y}$ of the $i$-skeleton in $|Y|$, so $\widehat{Y}_i$ is a closed $G$-invariant subspace of $\widehat{Y}$, and for each $i<d$, $\widehat{Y}_{i+1}\setminus \widehat{Y}_i$ is locally closed and $G$-invariant.  As in \cite[Line (2.6.33), page 115]{Kashiwara:1990aa}, there is an exact triangle
\begin{equation}\label{e t 1}
\xymatrix{ F_{\widehat{Y}_{i}} \ar[dr]|-\circ &&\ar[ll]_-r F_{\widehat{Y}_{i+1}} \\
& F_{\widehat{Y}_{i+1}\setminus \widehat{Y}_i} \ar[ur]_-i & }
\end{equation}
(where we use the circled arrow to denote a map of degree one) in $\dcc(\text{Sh}_G(\widehat{Y}))$; indeed (compare \cite[Example 10.4.9]{Weibel:1995ty}), this follows as the sequence of $G$-sheaves 
$$
0\to F_{\widehat{Y}_{i+1}\setminus \widehat{Y}_i}\stackrel{i}{\to} F_{\widehat{Y}_{i+1}}\stackrel{r}{\to} F_{\widehat{Y}_{i}}\to 0
$$  
is exact, which in turn follows directly from Definition \ref{f sub z}.  As derived functors preserve exact triangles (by definition: see for example [Definitions 10.2.6 and 10.5.1] or \cite[page 38 and Definition 1.8.1]{Kashiwara:1990aa}), we see that the triangle 
$$
\xymatrix{ R\Gamma_cF_{\widehat{Y}_{i}} \ar[dr]|-\circ &&\ar[ll] R\Gamma_cF_{\widehat{Y}_{i+1}} \\
& R\Gamma_cF_{\widehat{Y}_{i+1}\setminus \widehat{Y}_i} \ar[ur] & }
$$
is exact in $\dcc(G\text{-}k[G_{tor}])$.  For simplicity (and following \cite[page 115]{Kashiwara:1990aa}), write $H_c^n(F)$ for the $n^\text{th}$ cohomology group of the cochain complex $R\Gamma_cF$.  As taking cohomology is a cohomological functor (see \cite[Definition 10.2.7 and Corollary 10.1.4]{Weibel:1995ty} or \cite[Definition 1.5.2 and Proposition 1.5.6]{Kashiwara:1990aa}), for each $i$, we thus get a long-exact sequence
{\small \begin{equation}\label{coho les}
\cdots \to H^n_c(F_{\widehat{Y}_{i+1}\setminus \widehat{Y}_i}) \to H^n_c(F_{\widehat{Y}_{i+1}})\to H^n_c(F_{\widehat{Y}_i}) \to H^{n+1}_c(F_{\widehat{Y}_{i+1}\setminus \widehat{Y}_i})\to \cdots
\end{equation}}
of $G\text{-}k[G_{tor}]$ modules.  

We specialize now to the case $F=k_{\widehat{Y}}$; to avoid double-subscripts, for a locally closed $G$-invariant subspace $Z$ of $\widehat{Y}$, we write $k_{Z}$ for what should more properly be called $(k_{\widehat{Y}})_Z$.  Let 
\begin{equation}\label{base g r}
0\to k_{\widehat{Y}}\stackrel{d_F}{\to} F^0 \stackrel{d_F}{\to} F^1 \stackrel{d_F}{\to} \cdots
\end{equation}
be the Godement resolution of $k_{\widehat{Y}}$ as in Definition \ref{g r}.  Note that if $Z$ is a locally closed $G$-invariant subspace of $\widehat{Y}$ then 
\begin{equation}\label{g r ky}
0\to k_{Z}\to F^0_Z\to F^1_Z\to \cdots
\end{equation}
is also the Godement resolution of $k_Z$.

We aim now to compute $R\Gamma_ck_{\widehat{Y}_{i+1}\setminus \widehat{Y}_i}$.  First note that we have 
\begin{equation}\label{sim decom}
\widehat{Y}_{i+1}\setminus \widehat{Y}_i=\bigsqcup_{\sigma\in Y_{i+1}}G_\sigma\times \sigma^\circ
\end{equation}
where $\sigma^\circ$ is the (topological) open simplex corresponding to $\sigma$.   Using Lemma \ref{g c soft}, $R\Gamma_ck_{\widehat{Y}_{i+1}\setminus \widehat{Y}_i}$ can be taken to be 
$$
\Gamma_c(F^0_{\widehat{Y}_{i+1}\setminus \widehat{Y}_i})\to \Gamma_c(F^1_{\widehat{Y}_{i+1}\setminus \widehat{Y}_i})\to \Gamma_c(F^2_{\widehat{Y}_{i+1}\setminus \widehat{Y}_i}) \to \cdots ,
$$
where $(F^n_{\widehat{Y}_{i+1}\setminus \widehat{Y}_i})$ is the resolution in line \eqref{g r ky}.  On the other hand, line \eqref{sim decom} implies that for each $n$
\begin{equation}\label{g c deom}
\Gamma_c(F^n_{\widehat{Y}_{i+1}\setminus \widehat{Y}_i})=\bigoplus_{\sigma\in Y_{i+1}}\bigoplus_{g\in G_\sigma} \Gamma_c(\{g\}\times\sigma^\circ,F^n).
\end{equation}
On each of the subsets $\{g\}\times \sigma^\circ$ the corresponding complex 
\begin{equation}\label{single simp}
\Gamma_c(\{g\}\times\sigma^\circ,F^0)\to \Gamma_c(\{g\}\times\sigma^\circ,F^1)\to \Gamma_c(\{g\}\times\sigma^\circ,F^2)\to\cdots 
\end{equation}
computes the compactly supported cohomology of $\{g\}\times \sigma^\circ$ with coefficients in $k$.  Indeed, sheaf cohomology with compact supports and values in the sheaf of locally constant $k$-valued functions agrees with any of the classical\footnote{For example, singular, \v{C}ech, or Alexander-Spanier: see for example \cite[pages 242-244]{Hatcher:2002ud} for compactly supported singular cohomology.} definitions of compactly supported cohomology with coefficients in $k$ (at least for reasonably `nice' spaces like $\sigma^\circ$): see for example \cite[Chapter III]{Bredon:1997aa} or \cite[pages 98-100 and Chapter VIII]{Swan:1964aa}.   In other words, the homology of the sequence in line \eqref{single simp} is just the usual cohomology with compact supports and coefficients in $k$, as computed for example in \cite[Example 3.34]{Hatcher:2002ud}:
\begin{equation}\label{dim isos}
H_c^n(\{g\}\times \sigma^\circ,k)\cong \left\{\begin{array}{ll} k & n=i+1~(=\text{dim}(\sigma^\circ)) \\ 0 & \text{otherwise} \end{array}\right..
\end{equation}
Combining this with line \eqref{g c deom} implies that 
\begin{equation}\label{hn com}
H^n_c(k_{\widehat{Y}_{i+1}\setminus \widehat{Y}_i})\cong \left\{\begin{array}{ll} \bigoplus_{\sigma\in Y_i} k[G_\sigma] & n=i+1 \\ 0 & \text{otherwise} \end{array}\right.,
\end{equation}
where the right hand side is equipped with the $G$-action defined in line \eqref{g act}.  

Now, we deduce from the long exact sequences of line \eqref{coho les} that 
$$
H^n_c(k_{\widehat{Y}_i})\cong\left\{\begin{array}{ll} H^n_c(k_{\widehat{Y}}) & n<i \\ 0 & n>i  \end{array}\right.
$$
(and we ignore the case $n=i$).  Hence the long exact sequences of line \eqref{coho les} gives rise to a collection of short exact sequences
$$
0\to H^n_c(k_{\widehat{Y}})\to H^n _c(k_{\widehat{Y}_n})\to H^{n+1}_c(k_{\widehat{Y}_{n+1}\setminus \widehat{Y}_n})\to H_c^{n+1}(k_{\widehat{Y}_{n+1}})\to 0.
$$
Splicing these together we get a diagram
{\tiny \begin{equation}\label{big diag}
\xymatrix{ & & & & 0 & \cdots  &\\  
& & 0\ar[r]  & H_c^{n+1}(k_{\widehat{Y}}) \ar[r] & H_c^{n+1}(k_{\widehat{Y}_{n+1}}) \ar[r]  \ar[u] & H_c^{n+1}(k_{\widehat{Y}_{n+2}\setminus \widehat{Y}_{n+1}})  &    \\
& & 0 & & H_c^{n+1}(k_{\widehat{Y}_{n+1}\setminus \widehat{Y}_n}) \ar[u] \ar[ur]^-{d} & &\\
0 \ar[r] & H_c^{n-1}(k_{\widehat{Y}})\ar[r] & H_c^{n-1} (k_{\widehat{Y}_{n-1}}) \ar[u] \ar[r] & H_c^{n}(k_{\widehat{Y}_{n}\setminus \widehat{Y}_{n-1}}) \ar[ur]^-{d} \ar[r] & H_c^{n}(k_{\widehat{Y}_{n}}) \ar[r] \ar[u] &  0 & \\
& & H_c^{n-1}(k_{\widehat{Y}_{n-1}\setminus \widehat{Y}_{n-2}}) \ar[u] \ar[ur]^-{d} & & H_c^n(k_{\widehat{Y}}) \ar[u]  & & \\
\ar[r] & H_c^{n-2}(k_{\widehat{Y}_{n-2}\setminus \widehat{Y}_{n-3}}) \ar[ur]^-{d} \ar[r] & H_c^{n-2}(k_{\widehat{Y}_{n-2}}) \ar[u] \ar[r] & 0 & 0\ar[u]  & & \\
\cdots & & H_c^{n-2}(k_{\widehat{Y}}) \ar[u] & & & & \\
  & & 0 \ar[u] & & & & }
\end{equation}}
\noindent{}where the arrows labeled ``$d$'' are defined to make the diagram commute.  The diagonal line gives a sequence
\begin{equation}\label{rel com}
\cdots \stackrel{d}{\to} H^{n-1}_c(k_{\widehat{Y}_{n-1}\setminus \widehat{Y}_{n-2}}) \stackrel{d}{\to} H^{n}_c(k_{\widehat{Y}_{n}\setminus \widehat{Y}_{n-1}})\stackrel{d}{\to} H^{n+1}_c(k_{\widehat{Y}_{n+1}\setminus \widehat{Y}_{n}}) \stackrel{d}{\to} \cdots ;
\end{equation}
and from exactness of the rows and columns of the diagram in line \eqref{big diag} this is a complex.  A diagram chase based on line \eqref{big diag} shows that the homology of this complex is isomorphic to $H_c^{n}(k_{\widehat{Y}})$ in degree $n$.  However, we need the more precise statement that the complex in line \eqref{rel com} is isomorphic to $R\Gamma_ck_{\widehat{Y}}$ (equivalently, the complex in line \eqref{base g r}) in the derived category $\dcc(G\text{-}k[G_{tor}])$; to establish this, we proceed as follows (compare \cite[Exercise II.21]{Kashiwara:1990aa}, which gives a similar computation in a different context).

Define now $C^n:=\Gamma_c(F^n_{\widehat{Y}_n}) / \text{Image}(d_F\circ i)$, where $d_F: \Gamma_c(F^{n-1}_{\widehat{Y}_n})\to \Gamma_c(F^n_{\widehat{Y}_n})$ is induced from the differential  in the Godement resolution (see line \eqref{base g r}), and $i:\Gamma_c(F^{n-1}_{\widehat{Y}_n\setminus \widehat{Y}_{n-1}})\to \Gamma_c(F^{n-1}_{\widehat{Y}_n})$ is induced from the inclusion of sheaves in line \eqref{e t 1}.  We define a differential $d_C:C^n\to C^{n+1}$ as follows.  First, we note that the short exact sequence underlying the exact triangle in line \eqref{e t 1} induces a short exact sequence 
$$
0\to \Gamma_c(F^n_{\widehat{Y}_{n+1}\setminus \widehat{Y}_{n}})\stackrel{i}{\to}  \Gamma_c(F^n_{\widehat{Y}_{n+1}}) \stackrel{r}{\to} \Gamma_c(F^n_{\widehat{Y}_{n}})\to 0
$$
by \cite[Proposition 2.3.6 (v), and Proposition 2.5.8]{Kashiwara:1990aa}.  The differential $d_C$ is defined by taking an element in $C^n$, using this sequence to lift it to an element of $\Gamma_c(F^n_{\widehat{Y}_{n+1}})$, taking its image under the differential $d_F:\Gamma_c(F^n_{\widehat{Y}_{n+1}})\to \Gamma_c(F^{n+1}_{\widehat{Y}_{n+1}})$ arising from the Godement resolution in line \eqref{base g r}, and then taking the quotient map from $\Gamma_c(F^{n+1}_{\widehat{Y}_{n+1}})$ to $C^{n+1}$; one checks that this is well-defined, and does indeed produce a complex $(C^\bullet,d_C)$.  We have now morphisms of complexes
\begin{equation}\label{com com}
(F^\bullet_{\widehat{Y}},d_F) \stackrel{r}{\to} (C^\bullet,d_C) \stackrel{i}{\leftarrow} H^{\bullet}(k_{\widehat{Y}_{\bullet}\setminus \widehat{Y}_{\bullet-1}},d)
\end{equation}
where the right hand side is the complex in line \eqref{rel com}, the left hand side is the complex from line \eqref{base g r}, the right arrow is induced by the restriction maps $r:F^n_{\widehat{Y}}\to F^n_{\widehat{Y}_n}$ and the right hand map is induced from the inclusion maps $F^n_{\widehat{Y}_n}\leftarrow F^{n}_{\widehat{Y}_{n}\setminus \widehat{Y}_{n-1}}:i$.  One now checks by a(n admittedly quite involved) diagram chase based on the isomorphisms in line \eqref{dim isos}, the information in line \eqref{big diag}, and the definition of the boundary map in homology induced from a short exact sequence of complexes, that the maps in line \eqref{com com} are both quasi-isomorphisms, and therefore that $(F^\bullet_{\widehat{Y}},d_F)$ and $H^{\bullet}(k_{\widehat{Y}_{\bullet}\setminus \widehat{Y}_{\bullet-1}},d)$ are isomorphic in $\dcc(G\text{-}k[G_{tor}])$ as claimed.

To complete the proof, it remains to identify the complex in line \eqref{rel com} with the complex from Definition \ref{simp der def}.  The isomorphism in line \eqref{hn com} shows that the modules appearing match.  Finally, a computation based on the isomorphisms in line \eqref{single simp} and completely analogous to that computing the boundary maps in classical cellular homology (compare \cite[pages 140-141]{Hatcher:2002ud}) shows that the boundary maps for this complex are exactly the maps $\partial$ for the complex from Definition \ref{simp der def}.  
\end{proof}

\subsection{Step 2: computation of $R\Gamma_ck_{\widehat{X}}$ for $X$ finite-dimensional}\label{s2}

In this step, our aim is a lemma showing that $R\Gamma_ck_{\widehat{X}}$ can be computed from a complex with good properties. 

We need two preliminary lemmas, which are no doubt well-known to experts.

\begin{lemma}\label{tf lem}
Let $Z$ be a topological space, let $k$ be an integral domain, and let $F$ be a sheaf of $k$-modules over $Z$.  Then the following are equivalent:
\begin{enumerate}[(i)]
\item \label{tf1} for each open $U\subseteq Z$, $F(U)$ is torsion free;
\item \label{tf2} each stalk $F_z$ is torsion free.
\end{enumerate}
Moreover, both imply the following:
\begin{enumerate}
\item[(iii)] \label{tf3} for each open $U\subseteq Z$, $\Gamma_c(U,F)$ is torsion free.
\end{enumerate} 
\end{lemma}

\begin{proof}
That \eqref{tf1} implies \eqref{tf2} follows as $F_z=\lim_{U\owns z} F(U)$ (by definition), and as direct limits of torsion free modules are torsion free (this follows from a direct check that we leave to the reader).  That \eqref{tf2} implies \eqref{tf1} follows as $F(U)$ can be realized as a submodule of $\prod_{z\in U}F_z$ (compare Definition \ref{g r}).  They both imply the last point as $\Gamma_c(U,F)$ is a submodule of $F(U)$.
\end{proof}

\begin{lemma}\label{l hom lem}
Let $Z$ be a locally compact Hausdorff space and let $k$ be a commutative ring with unit.  Let $F$ be a sheaf of $k$-modules over $Z$ such that each stalk $F_z$ is torsion free.  Let $\pi:\mathcal{F}\to Z$ be the corresponding \'{e}tale space of $F$.  Let $r\in k$ be non-zero, and write $m_r:\mathcal{F}\to \mathcal{F}$ be the map which multiplies by $r$ on each stalk.  Then $m_r$ is a local homeomorphism.  
\end{lemma}

\begin{proof}
The map $m_r:\mathcal{F}\to \mathcal{F}$ is continuous.  Moreover, as each $F_z$ is torsion free $m_r$ is injective when restricted to each stalk and thus globally injective.  It thus suffices to show that for any $f\in \mathcal{F}$ there is an open neighbourhood $U\owns f$ such that $m_r(U)$ is open, and such that 
$m_r:U\to m_r(U)$ is a homeomorphism between open subsets of $\mathcal{F}$.

Note that $\mathcal{F}$ is locally homeomorphic to $Z$ via the structure map $\pi:\mathcal{F}\to Z$, and that any section is a local homeomorphism (compare for example \cite[page 4]{Bredon:1997aa} for this and other basic properties used below).   

Let $f$ be an element of the stalk $F_z$, and let $s:U_0\to \mathcal{F}$ be a section with $U_0\owns z$ open, $s:U_0\to s(U_0)$ a homeomorphism, and $s(z)=f$.  Note that $m_r\circ s:U_0\to \mathcal{F}$ is also a section, whence a local homeomorphism; hence there exists open $U_1\subseteq U_0$ such that $U_1\owns z$ and such that $m_r\circ f:U_1\to m_r(f(U_1))$ is a homeomorphism.  Setting $U=f(U_1)$ we are done.
\end{proof}

Here is the main lemma from this step that will enable us to give a good model for $R\Gamma_ck_{\widehat{X}}$.  Part \eqref{rgcx 1} (at least) is likely well-known: compare for example \cite[3.3]{CrainicMoerdijk2000aa}.

\begin{lemma}\label{rgcx}
Let $Z$ be a locally compact, Hausdorff $G$-space with finite $c$-$k$-homological dimension (see Definition \ref{cdim} above).  Let $k$ be a commutative unital ring, and let $k_Z$ denote the $G$-sheaf of locally constant $k$-valued functions on $Z$.  Then there exists a finite-length resolution 
$$
0\to k_Z\to I^0\to \cdots \to I^m \to 0
$$
of $k_{Z}$ in the category of $G$-sheaves over $Z$ with the following properties:
\begin{enumerate}[(i)]
\item \label{rgcx 1} each $I^i$ is a $c$-soft $G$-sheaf;
\item \label{rgcx 2} if $k$ is moreover a Pr\"{u}fer domain (see Remark \ref{k rem} above) then for each open set $U\subseteq Z$, the space of compactly supported sections $\Gamma_c(U,I^i)$ is flat as a $k$-module.
\end{enumerate}
\end{lemma}

\begin{proof}
Let $0\to k_Z\to I^0\to \cdots \to I^{m-1}$ be the first part of the Godement resolution (see Definition \ref{g r}) of $k_Z$, and let $I^m:=I^{m-1}/\text{image}(I^{m-2})$ be the corresponding quotient $G$-sheaf.  We claim that  the resolution 
\begin{equation}\label{finite gr}
0\to k_Z\to I^0\to \cdots \to I^m \to 0
\end{equation}
of $G$-sheaves has the desired properties.  

First, let us look at $c$-softness.  Lemma \ref{g c soft} part \eqref{gcs god} implies that all of $I^0$, ..., $I^{m-1}$ are $c$-soft.  To show that $I^m$ is also $c$-soft, note that the equivalence of (a) and (b) in \cite[Theorem 16.2]{Bredon:1997aa} then implies that if 
$$
0\to k_Z\to F^0\to \cdots \to F^m \to 0
$$
is any resolution with $F^0$,...,$F^{m-1}$ $c$-soft, then $F^m$ is automatically $c$-soft.  Hence the resolution in line \eqref{finite gr} does indeed consist of $c$-soft $G$-sheaves.  

Now let us look at the claimed flatness statement.  As $k$ is a Pr\"{u}fer domain, a $k$-module is flat if and only if it is torsion free.  It thus suffices to show that $\Gamma_c(U,I^i)$ is a torsion free $k$-module for all open $U\subseteq Z$, and all $i$.  For simplicity, let us say that a sheaf $F$ of $k$-modules over $Z$ is \emph{torsion-free} if each stalk $F_z$ is torsion-free as a $k$ module.  Using Lemma  Lemma \ref{tf lem}, it suffices to show that each $I^i$ is torsion free.

Clearly $k_Z$ is torsion free.  Thanks to this and the construction of the resolution in line \eqref{finite gr}, it will suffices to show the following fact: if $F$ is torsion free, and $F_{disc}$ is the corresponding sheaf of not-necessarily continuous sections (see Definition \ref{g r}), then both $F_{disc}$ and $F_{disc}/F$ are torsion free.  For $F_{disc}$, this is clear, so it will suffice to show that for each $z\in Z$, the stalk $(F_{disc}/F)_z=(F_{disc})_z/F_z$ is torsion free.  Fixing non-zero $r\in k$, it suffices to show that if $f\in (F_{disc})_z$ satisfies $rf\in F_z$, then $f\in F_z$.  Indeed, let $U\owns z$ be an open neighbourhood of $z$ such that there is a section $t\in F(U)$ extending $rf$ and a section $s\in F_{disc}(U)$ extending $f$; we think of both $s$ and $t$ as functions $U\to \mathcal{F}$ from $U$ to the \'{e}tale space underlying $F$ that are sections of the canonical quotient $\pi:\mathcal{F}\to Z$, and such that $t$ is continuous, but $s$ may not be.  Now, by definition of the direct limit, there exists open $V\subseteq U$ with $V\owns z$ and $rs=t$ in $V$.  Note that $t:V\to \mathcal{F}$ is a local homeomorphism, so shrinking $V$ further, we may assume that $t|_V$ is a homeomorphism onto its image.  Letting $m_r:\mathcal{F}\to \mathcal{F}$ be the local homeomorphism from Lemma \ref{l hom lem}, there is open $W\owns t(z)$ in $\mathcal{F}$ such that $m_r:m_r^{-1}(W)\to W$ is a homeomorphism.  Replacing $V$ with $U':=V\cap t^{-1}(m_r^{-1}(W))$, we find that $s=m_r^{-1}\circ t$ on $U'$.  As $m_r^{-1}$ and $t$ are continuous as functions $U'\to \mathcal{F}$, $s$ is too, and we are done.
\end{proof}

\subsection{Step 3: computation of $\text{Ext}_{G\text{-}k[G_{tor}]}(R\Gamma_ck_{\widehat{Y}},R\Gamma_ck_{\widehat{X}})$}\label{s3}

In this step, we show that if $Y$ is a simplicial complex satisfying the conditions from Definition \ref{simp der def} and $X$ has finite $c$-$k$-cohomological dimension, then the group $H_{G,k,c}(|Y|,X)=\text{Ext}_{G\text{-}k[G_{tor}]}(R\Gamma_ck_{\widehat{Y}},R\Gamma_ck_{\widehat{X}})$ can be computed as the homology of an appropriate double complex.

We need some preliminary lemmas.  The first lemma would be almost immediate if $k[G_{tor}]$ were unital, but this is not always the case; compare \cite[Remark 2.9]{BonickeDellAieraGabeWillett2023aa} for essentially the same point.

\begin{lemma}\label{ep}
The category of nondegenerate $G$-$k[G_{tor}]$ modules has enough projective objects.  
\end{lemma}

\begin{proof}
Let $R=k[G_{tor}]\rtimes G$ be the algebraic crossed product of $k[G_{tor}]$ by $G$: precisely, elements of $R$ are formal sums $\sum_{g\in G} a_g g$ where $a_g\in k[G_{tor}]$, only finitely many $a_g$ are non-zero, and multiplication is determined by the multiplication rules in $k[G_{tor}]$ and $G$ together with the relation $g(a)g^{-1}=\alpha_g(a)$ for $a\in k[G_{tor}]$ and $\alpha:G\to \text{Aut}(k[G_{tor}])$ the action.  Then the category of nondegenerate $G$-$k[G_{tor}]$ modules identifies with the category of $R$-modules, so it suffices to show that the latter has enough projectives.  As any nondegenerate $R$-module is the quotient of a free $R$-module, it suffices to show that free $R$-modules are projective.  

Let then $F=\bigoplus_{i\in I} R$ be a free module, let $\pi:M\to N$ be a surjection, and let $\phi:F\to N$ be any map, so we want to fill in the diagram below so as to make it commute
$$
\xymatrix{ & M \ar[d]^-{\pi} \\ F \ar[r]^-\phi \ar@{-->}[ur] & N }.
$$
Thanks to the universal property of the direct sum, we may assume $I$ consists of a singleton, so $F=R$.  For each $g\in G_{tor}$, let $\chi_g\in R$ denote the characteristic function of $g$.  Then by nondegenerary, $M=\bigoplus_{g\in G_{tor}}\phi(\chi_g) M$.  Moreover, $R=\bigoplus_{g\in G_{tor}} R\chi_g$ (as a left $R$-module).  For each $g\in G_{tor}$ choose a lift $m_g\in M$ of $\phi(\chi_g)$.  Define 
$$
\widetilde{\phi}:R\to M,\quad r\mapsto \sum_{g\in G_{tor}} r\chi_ga_g
$$
(the sum is finite as $r\chi_g\neq 0$ for only finitely many $g$).  This then is a lift as required.
\end{proof}

\begin{lemma}\label{adj fun}
Let $S$ be a $G$-set such that for each $s\in S$, the stabilizer $G_s$ is finite.   Let $k$ be a commutative unital ring.  Let $\bigoplus_{s\in S}k[G_{s}]$ be made into a $G$-$k[G_{tor}]$ module by the $G$-action from line \eqref{g act}, and the pointwise multiplication action of $k[G_{tor}]$.  Let $\chi_{G_{s}}$ be the characteristic function of $G_s$.  Then for any $G$-$k[G_{tor}]$ module $M$ there is a canonical isomorphism
$$
\text{Hom}_{G\text{-}k[G_{tor}]}\Bigg(\bigoplus_{s\in S}k[G_{s}],M\Bigg)\cong \Bigg(\prod_{s\in S} \chi_{G_s}M\Bigg)^G. 
$$
If moreover $S/G$ is finite, there is a canonical isomorphism
$$
\Bigg(\prod_{s\in S} \chi_{G_s}M\Bigg)^G\cong  \Bigg(\bigoplus_{s\in S}\chi_{G_{s}}M\Bigg)_G
$$
where ``$-_G$'' means taking coinvariants for the $G$-action on $\bigoplus_{s\in S}\chi_{G_{s}}M$ defined via the formula in line \eqref{g act}.
\end{lemma}

\begin{proof}
We first construct a homomorphism
\begin{equation}\label{iso1}
\Bigg(\prod_{s\in S} \chi_{G_s}M\Bigg)^G\to \text{Hom}_{G\text{-}k[G_{tor}]}\Bigg(\bigoplus_{s\in S}k[G_{s}],M\Bigg).
\end{equation}
by sending an element $(m_s)_{s\in S}$ of the left hand side to a homomorphism on the right hand side via the formula $(a_s)_{s\in S}\mapsto \sum_s a_sm_s$.  This homomorphism is actually an isomorphism: given $\phi$ on the right hand side, the inverse is given by sending $\phi$ to $(\phi(\chi_{G_s}))_{s\in S}$.  

On the other hand if $\alpha$ denotes the $G$-action on $\bigoplus_{s\in S}\chi_{G_{s}}M$, then one can define a map
$$
\bigoplus_{s\in S}\chi_{G_{s}}M\to \Bigg(\prod_{s\in S} \chi_{G_s}M\Bigg)^G,\quad (a_s)_{s\in S}\mapsto \Bigg(\sum_{\{g\in G,t\in S\mid gt=s\}}\alpha_g(a_t)\Bigg)_{s\in S}.
$$
One checks that this descends to well-defined homomorphism 
\begin{equation}\label{iso2}
\Bigg(\bigoplus_{s\in S}\chi_{G_{s}}M\Bigg)_G\to \Bigg(\prod_{s\in S} \chi_{G_s}M\Bigg)^G.
\end{equation}
Moreover this is an isomorphism: the inverse is given by choosing a (finite!) set of orbit representatives $T\subseteq S$, and mapping $(m_s)_{s\in S}$ to the element of $\bigoplus_{s\in S}\chi_{G_{s}}M$ whose entries are $m_t$ for $t\in T$, and zero otherwise, and then passing through the quotient map to the coinvariants.  Combining the isomorphisms in lines \eqref{iso1} and \eqref{iso2}, we are done.
\end{proof}

The next lemma is the first place (of two) in the proof of Proposition \ref{bcr and gh} where the assumption that $k$ contains inverses of the orders of (appropriate) torsion elements of $G$ is used.  

\begin{lemma}\label{proj 1}
Let $S$ be a $G$-set such that for each $s\in S$, the stabilizer $G_s$ is finite.   Let $k$ be a commutative unital ring that contains $|G_s|^{-1}$ for each $s$.  Let $\bigoplus_{s\in S}k[G_{s}]$ be made into a $G$-$k[G_{tor}]$ module by the $G$-action from line \eqref{g act}, and the pointwise multiplication action of $k[G_{tor}]$.  Then $\bigoplus_{s\in S}k[G_{s}]$ is projective in the category of $G$-$k[G_{tor}]$ modules.
\end{lemma}

\begin{proof}
Lemma \ref{adj fun} implies that for any $G$-$k[G_{tor}]$ module 
$$
\text{Hom}_{G\text{-}k[G_{tor}]}\Bigg(\bigoplus_{s\in S}k[G_{s}],M\Bigg)\cong \Bigg(\prod_{s\in S}\chi_{G_{s}}M\Bigg)^G.
$$
Let $T\subseteq S$ be a set of $G$-orbit representatives.  Then one checks that there is an isomorphism
$$
\Bigg(\prod_{s\in S}\chi_{G_{s}}M\Bigg)^G\cong \prod_{s\in T}(\chi_{G_{s}}M)^{G_s}, \quad (a_s)_{s\in S}\mapsto (a_t)_{t\in T}.
$$
An object $P$ in an abelian category is projective if and only if the functor $N\mapsto \text{Hom}(P,N)$ is exact (see for example \cite[Lemma 2.2.3]{Weibel:1995ty}), whence it suffices to show that the functor $M\mapsto \prod_{s\in T}(\chi_{G_{s}}M)^{G_s}$ from $G$-$k[G_{tor}]$ modules to abelian groups is exact; for that, it will suffice to show that the functor $N\mapsto N^{G_s}$ from $G_s$-$k$ modules to abelian groups is exact.  For this, note that if $p:=\frac{1}{|G_s|}\sum_{g\in G_s}g\in \text{End}(N)$, then $p$ is an idempotent and $pN=N^{G_s}$; exactness follows.
\end{proof}

We are now ready to define the double complex we will use to compute $H_{G,k,c}^*(|Y|,X)$.  We will need the complex below in cases where $Y$ is not $G$-finite, so state it in that level of generality.

\begin{definition}\label{dc1}
Let $k$ be a commutative unital ring.  Let $Z$ be a proper $G$-oriented, type preserving $G$-simplicial complex.  Let $X$ be a locally compact Hausdorff $G$-space with finite $c$-$k$-cohomological dimension, and note that this implies that the space $\widehat{X}$ of Definition \ref{z hat} has finite $c$-$k$-cohomological dimension too, so there exists a resolution 
\begin{equation}\label{x res}
0\to k_{\widehat{X}}\to I^0\to \cdots \to I^m \to 0
\end{equation}
of $c$-soft sheaves as in Lemma \ref{rgcx}\footnote{We do not assume that $k$ is a Pr\"{u}fer domain, so may ignore the condition in part \eqref{rgcx 2} of Lemma \ref{rgcx}.}.  

Let $\phi_X:\widehat{X}\to G_{tor}$ be as in Definition \ref{z hat}, for each simplex $\sigma\in Y_i$ let $G_\sigma$ be the stabilizer, and define $X_{\sigma}:=\phi_X^{-1}(G_\sigma)$.  Define moreover 
$$
C^{p,q}:=\left\{\begin{array}{ll} \oplus_{\sigma\in Z_{-p}}\Gamma_c(X_\sigma,I^q) & p\leq 0,~0\leq q\leq m \\ 0 & \text{otherwise} \end{array}\right.
$$
We may form the modules $(C^{p,q})$ into a double cochain complex via the following differentials.

Horizontal differentials.  For each $p\geq 0$, each simplex $\sigma\in Z_p$ and each $j\in \{0,...,p\}$, let $\sigma^{(j)}\in Z_{p-1}$ denote the $j^\text{th}$ face of $\sigma$, i.e.\ the simplex that has the same vertices as $\sigma$ except the $j^\text{th}$ face is removed.  Note that the type-preserving assumption implies that $G_\sigma\subseteq G_{\sigma^{(j)}}$ for each $j$, whence for each such $j$ and each $q\in \{0,...,m\}$ there is a canonical inclusion $\Gamma_c(X_\sigma,I^q)\subseteq \Gamma_c(X_{\sigma^{(j)}},I^q)$.  Define 
$$
\partial_j:\underbrace{\bigoplus_{\sigma\in Z_p}\Gamma_c(X_\sigma,I^q)}_{=C^{-p,q}}\to \underbrace{\bigoplus_{\sigma\in Z_{p-1}}\Gamma_c(X_\sigma,I^q)}_{=C^{-p+1,q}}
$$
by sending an element $a$ supported in $\Gamma_c(X_\sigma,I^q)$ to its image under the corresponding inclusion $\Gamma_c(X_\sigma,I^q)\subseteq \Gamma_c(X_\sigma^{(j)},I^q)$ and extending by linearity.  The horizontal differential is then
$$
\partial^{(h)}:C^{p,q}\to C^{p+1,q}\quad \partial:=(-1)^q\sum_{j=0}^i (-1)^j\partial_j
$$
(the ``$(-1)^q$'' is to ensure that $\partial^{(h)}$ anticommutes with the vertical differential $\partial^{(v)}$ defined below).

Vertical differentials.  The vertical differentials $\partial^{(v)}:C^{p,q}\to C^{p,q+1}$ are those functorially induced by the differentials from the resolution in line \eqref{x res} above on each summand $\Gamma_c(X_\sigma,I^q)$.
\end{definition}

\begin{lemma}\label{hom cor}
Let $Y$ be a proper, $G$-finite, type-preserving, $G$-oriented $G$-simplicial complex of dimension $d$.   Let $X$ be a locally compact, Hausdorff $G$-space with finite $c$-$k$-cohomological dimension.  Let $k$ be a commutative unital ring such that $n^{-1}$ exists in $k$ whenever $n$ is the order of a torsion element of $G$ that fixes a point $x\in X$.  Let 
\begin{equation}\label{y x dc}
\xymatrix{ & 0 & & 0 & \\ 0 \ar[r] & C^{-d,m} \ar[u] \ar[r] & \cdots  \ar[r] & C^{0,m} \ar[u] \ar[r] & 0  \\
& \vdots \ar[u] & & \vdots \ar[u] &  \\
0 \ar[r] & C^{-d,0} \ar[r] \ar[u] & \cdots  \ar[r] & C^{0,0} \ar[u] \ar[r] & 0 \\
& 0 \ar[u] & & 0 \ar[u] &}
\end{equation}
denote the double cochain complex from Definition \ref{dc1} above.  

Then $H_{G,k,c}^*(|Y|,X)$ is canonically isomorphic to the cohomology of the (direct sum) total complex one gets by starting with the double complex in line \eqref{y x dc} above and taking coinvariants in each entry, i.e.\ to the cohomology of the cochain complex $(C^n,\partial)_{n=-d}^m$ with 
$$
C^n:=\bigoplus_{p+q=n} (C^{p,q})_G
$$
and $\partial$ the differential induced on coinvariants by $\partial^{(h)}+\partial^{(v)}$.  
\end{lemma}

\begin{proof}
Using Lemma \ref{g-fin}, $R\Gamma_ck_{\widehat{Y}}$ is isomorphic in $\dcc(G\text{-}k[G_{tor}])$ to the complex from Definition \ref{simp der def}; we identify it with that complex.   Let $S_X:=\{g\in G_{tor}\mid gx=x\text{ for some } x\in X\}$.  Then there is a canonical splitting of the category of $G\text{-}k[G_{tor}]$ modules into a product of the categories of $G$-$k[S_X]$-modules and $G$-$k[G_{tor}\setminus S_X]$ modules; this gives rise to a splitting of $\dcc(G\text{-}k[G_{tor}])$ as a corresponding product of categories.  As $R\Gamma_ck_{\widehat{X}}$ lives in $\dcc(G\text{-}k[S_X])$, we may work inside this category, and replace $R\Gamma_ck_{\widehat{Y}}$ with its image in $\dcc(G\text{-}k[S_X])$; in other words, the module 
$$
\bigoplus_{\sigma\in Y_i}k[G_{\sigma}]
$$
from Definition \ref{simp der def} gets replaced by 
$$
\bigoplus_{\sigma\in Y_i}k[G_{\sigma}\cap S_X].
$$
Having made this adjustment, all the modules appearing in $R\Gamma_ck_{\widehat{Y}}$ are projective by our assumption on $k$ and Lemma \ref{proj 1}. 

Thanks to Lemma \ref{ep}, the category $G$-$k[G_{tor}]$ has enough projectives.  Hence from \cite[Theorem 10.7.4]{Weibel:1995ty} and following the conventions in \cite[2.7.4]{Weibel:1995ty} we may compute $H_{G,k,c}^*(|Y|,X)=\text{Ext}^*_{G\text{-}k[S_X]}(R\Gamma_ck_{\widehat{Y}},R\Gamma_ck_{\widehat{X}})$ by following the steps below:
\begin{enumerate}
\item find a cochain complex quasi-isomorphic to $R\Gamma_ck_{\widehat{Y}}$ that consists of projective objects (this has already been done by the adjustment above);
\item convert $R\Gamma_ck_{\widehat{Y}}$ to a chain complex by replacing the $p^{\text{th}}$ module by the $(-p)^\text{th}$;
\item take the Hom double complex $H^{p,q}$, where the $(p,q)^\text{th}$ module is the homomorphisms in $G$-$k[S_X]$ from the $p^\text{th}$ element of the chain complex associated to $R\Gamma_ck_{\widehat{Y}}$ to the $q^\text{th}$ element of the cochain complex $R\Gamma_ck_{\widehat{X}}$;
\item take the (direct product) total complex of the associated double complex, and take (co)homology. 
\end{enumerate}

Carrying out this process above, we see that 
$$
H^{p,q}=\left\{\begin{array}{ll} \text{Hom}_{G\text{-}k[S]} \Bigg( \bigoplus_{\sigma\in Y_{-p}}k[G_{\sigma}\cap S_X],\Gamma_c(\widehat{X},I^q)\Bigg), & -d\leq p\leq 0,~0\leq q\leq m \\ 0, & \text{otherwise} \end{array}\right..
$$
On the other hand, Lemma \ref{adj fun} implies then that for $-d\leq p\leq 0$ and $0\leq q\leq m$
$$
H^{p,q}=\Bigg(\bigoplus_{\sigma\in Y_{-p}}\Gamma_c(X_\sigma,I^q)\Bigg)_G
$$
and is zero otherwise.  Moreover, a direct check shows that the boundary maps that arise from Definition \ref{dc1} by taking coinvariants and those arrived at functorially from the process above agree.  Hence the double complex we arrive at agrees with that from line \eqref{y x dc}.  Finally, we note that as this double complex is of finite extent, the direct product and direct sum total complexes agree, and we have the statement.
\end{proof}

\subsection{Step 4: computation of $H^*_{G,!}(Z,X)$}\label{s4}

Recall from Definition \ref{bs hom} (the original source is \cite[page 316]{Baum:2002aa}) that for a general proper $G$-space $Z$, Baum and Schneider define 
$$
H^*_{G,k,!}(Z,X)=\lim_{\to} H^*_{G,k,c}(Y,X)
$$
where the limit is taken over all $G$-compact and $G$-invariant subspaces $Y$ of $Z$ ordered by inclusion.  In particular, if $Z$ is a proper $G$-simplicial complex, then the limit may be taken over all (geometric realizations of) $G$-finite subcomplexes.  As taking homology commutes with direct limits, we see that if $Z$ is a $G$-simplicial complex satisfying the assumptions of Lemma \ref{hom cor} with the exception of $G$-finiteness, then $H^*_{G,k,!}(|Z|,X)$ can be computed via the complex from Corollary \ref{hom cor}, i.e.\ we have the following result.

\begin{corollary}\label{! cor}
Let $Z$ be a proper, type-preserving, $G$-oriented $G$-simplicial complex.   Let $X$ be a locally compact, Hausdorff $G$-space with finite $c$-$k$-cohomological dimension.  Let $k$ be a commutative unital ring that contains $n^{-1}$ whenever $n$ is the order of a torsion element of $G$ that fixes a point $x\in X$.  Let 
\begin{equation}\label{y x dc 2}
\xymatrix{ \cdots & 0 & 0 & 0 & \\ \cdots  \ar[r] & C^{-2,m} \ar[u] \ar[r] & C^{-1,m}  \ar[r] \ar[u] & C^{0,m} \ar[u] \ar[r] & 0  \\
\cdots & \vdots \ar[u] & \vdots \ar[u] & \vdots \ar[u] &  \\
 \cdots  \ar[r] & C^{-2,0} \ar[u] \ar[r] & C^{-1,0}  \ar[r] \ar[u] & C^{0,0} \ar[u] \ar[r] & 0 \\
\cdots & 0 \ar[u] & 0\ar[u] & 0 \ar[u] &}
\end{equation}
denote the double cochain complex from Definition \ref{dc1} above.  

Then $H_{G,k,!}^*(|Z|,X)$ is canonically isomorphic to the cohomology of the (direct sum) total complex one gets by starting with the double complex in line \eqref{y x dc 2} above and taking coinvariants in each entry, i.e.\ to the cohomology of the cochain complex $(C^n,\partial)_{n=-\infty}^m$ with 
$$
C^n:=\bigoplus_{p+q=n} (C^{p,q})_G
$$
and $\partial$ the differential induced on coinvariants by $\partial^{(h)}+\partial^{(v)}$.  \qed
\end{corollary}

\subsection{Step 5: computation of $H_*(G,\Gamma_c(I^{-\bullet}))$}\label{s5}

The next lemma is no doubt well-known.  It is the second (and last) place in the proof of Proposition \ref{bcr and gh} where the assumption of $k$ containing inverses of the orders of (appropriate) finite subgroups of $G$ is used.  For the proof, recall that we write $kH$ for the group ring of a group $H$ over $k$.

\begin{lemma}\label{co inv}
Let $H$ be a finite group, let $k$ be a commutative unital ring in which $|H|$ is invertible, and let $V$ be an $H$-$k$-module which is flat as a $k$-module.  Then $V$ is flat as an $H$-$k$-module.
\end{lemma}

\begin{proof}
We need to show that the functor $M\mapsto M\otimes_{kH} V$ from $H$-$k$-modules to $k$-modules preserves exact sequences (as $k$ is commutative, we may consider $M\otimes_{kH} V$ as a $k$-module, and we do not need to worry about distinctions between left and right modules).   Let $\mu$ denote the $H$-action on $M$, and let $\nu$ denote the $H$-action on $V$.  Then $M\otimes_{kH}V$ is isomorphic to the quotient of $M\otimes_k V$ by the submodule $N$ generated by elements of the form $\mu_h(m)\otimes v-m\otimes \nu_h(v)$, or equivalently, by those of the form $m\otimes v - \mu_{h^{-1}}(m)\otimes \nu_h(v)$.  Define 
$$
p:=\frac{1}{|H|}\sum_{h\in H} \mu_{h^{-1}}\otimes \nu_h\in \text{End}_k(M\otimes_k V).
$$
Then one checks directly that $p$ is an idempotent, and that $(1-p)(M\otimes_k V)=N$.  Hence there are canonical isomorphisms:
$$
p(M\otimes_kV)\cong (M\otimes_k V)/N\cong M\otimes_{kH}V
$$
Using that $V$ is flat as a $k$-module, the assignment $M\mapsto p(M\otimes_k V)$ is an exact functor, so we are done.
\end{proof}

The next lemma is again no doubt well-known.    For the statement, recall the definition of induced module from Definition \ref{ind def} above.

\begin{lemma}\label{g-mod proj}
Let $k$ be a commutative unital ring.  Let $H$ be a subgroup of a group $G$, and $M$ be an $H$-$k$-module.  If $M$ is free (respectively projective, respectively flat) then so too is $\text{Ind}_H^G(M)$.
\end{lemma}

\begin{proof}
First let $M\cong \bigoplus_{i\in I} kH$ be free. Then $\text{Ind}_H^G\cong \bigoplus_i kG$ is also free.  If $M$ is projective, then $M\oplus N$ is free for some $H$-$k$-module $N$, whence $\text{Ind}_H^G(M\oplus N)=\text{Ind}_H^G(M)\oplus \text{Ind}_H^G(N)$ is free, so $\text{Ind}_H^G(M)$ is projective.  Finally, if $M$ is flat then $M$ is a direct limit of (finitely generated) free modules by \cite[Th\'{e}or\`{e}me 1.2]{Lazard:1969aa}; as $\text{Ind}_H^G$ commutes with direct limits, it is a direct limit of such modules too by the first part, whence flat.
\end{proof}

We need an analogue of the complex in Definition \ref{dc1}.

\begin{definition}\label{dc2}
Let $k$ be a commutative unital ring.  Let $Z$ be a proper $G$-oriented, type preserving $G$-simplicial complex.  Let $X$ be a locally compact Hausdorff $G$-space with finite $c$-$k$-cohomological dimension, and note that this implies that the space $\widehat{X}$ of Definition \ref{z hat} has finite $c$-$k$-cohomological dimension too, so there exists a resolution 
\begin{equation}\label{x res 2}
0\to k_{\widehat{X}}\to I^0\to \cdots \to I^m \to 0
\end{equation}
of $c$-soft sheaves as in Lemma \ref{rgcx}\footnote{We do not assume that $k$ is a Pr\"{u}fer ring, so may ignore the condition in part \eqref{rgcx 2} of Lemma \ref{rgcx}.}.

Let $\phi_X:\widehat{X}\to G_{tor}$ be as in Definition \ref{z hat} and for each simplex $\sigma\in Y_i$, let $G_\sigma$ be the stabilizer, and define $X_{\sigma}:=\phi_X^{-1}(G_\sigma)$.  Define moreover 
$$
D_{p,q}:=\left\{\begin{array}{ll} \bigoplus_{\sigma\in Z_{p}}\Gamma_c(X_\sigma,I^{-q}), & p\geq 0,~-m\leq q\leq 0\\
 \Gamma_c(\widehat{X},I^{-q}), & p=-1,~ -m\leq q\leq 0 \\ 0, & \text{otherwise} \end{array}\right.
$$
We may form the modules $(D_{p,q})$ into a double chain complex via the following differentials.

Horizontal differentials.  For each $p\geq 1$, each simplex $\sigma\in Z_p$ and each $j\in \{0,...,p\}$, let $\sigma^{(j)}\in Z_{p-1}$ denote the $j^\text{th}$ face of $\sigma$, i.e.\ the simplex that has the same vertices as $\sigma$ except the $j^\text{th}$ face is removed.  Note that the type-preserving assumption implies that $G_\sigma\subseteq G_{\sigma^{(j)}}$ for each $j$, whence for each such $j$ and each $q\in \{0,...,m\}$ there is a canonical inclusion $\Gamma_c(X_\sigma,I^q)\subseteq \Gamma_c(X_{\sigma^{(j)}},I^q)$.  Define 
$$
\partial_j:\underbrace{\bigoplus_{\sigma\in Z_p}\Gamma_c(X_\sigma,I^q)}_{=D_{p,-q}}\to \underbrace{\bigoplus_{\sigma\in Z_{p-1}}\Gamma_c(X_\sigma,I^q)}_{=D_{p-1,-q}}
$$
by sending an element $a$ supported in $\Gamma_c(X_\sigma,I^q)$ to its image under the corresponding inclusion $\Gamma_c(X_\sigma,I^q)\subseteq \Gamma_c(X_{\sigma^{(j)}},I^q)$ and extending by linearity.  The horizontal differential is then
$$
\partial^{(h)}:D_{p,q}\to D_{p-1,q}\quad \partial:=(-1)^q\sum_{j=0}^i (-1)^j\partial_j
$$
(the ``$(-1)^q$'' is to ensure that $\partial^{(h)}$ and $\partial^{(v)}$ as defined below anticommute).  For $p=0$, the differential 
$$
\partial:\underbrace{\bigoplus_{\sigma\in Z_0}\Gamma_c(X_\sigma,I^q)}_{=D_{p,q}}\to \underbrace{\Gamma_c(\widehat{X},I^q)}_{=D_{-1,q}}
$$
is defined by summing up the entries via the canonical inclusions $\Gamma_c(X_\sigma,I^q)\to \Gamma_c(\widehat{X},I^q)$, and again multiplying by $(-1)^q$.

Vertical differentials.  The vertical differentials $\partial^{(v)}:D_{p,q}\to D_{p,q-1}$ are those functorially induced by the differentials from the resolution in line \eqref{x res 2} above on each summand $\Gamma_c(X_\sigma,I^{-q})$, or on $\Gamma_c(\widehat{X},I^{-q})$ (note that we have converted cochain differentials to chain differentials by replacing $I^{\bullet}$ with $I^{-\bullet}$).
\end{definition}

We need one last lemma to relate Corollary \ref{! cor} to group homology.  We need a definition\footnote{It is analogous to \cite[page 56]{Schneider:1998aa}, but our conventions are different: to be consistent with the rest of this note, we take the \emph{(abstract) simplicial complex} consisting of all finite \emph{subsets} of the set $V_0$ defined below; in the above reference Schneider uses the \emph{simplicial set} of all finite \emph{sequences} from $V_0$; these two constructions lead to the same geometric realization.} first.

\begin{definition}\label{simp eg}
Let $\text{Fin}(G)$ denote the set of finite subgroups of $G$.  Let $V_0:=\bigsqcup_{H\in \text{Fin}(G)}G/H$, and let $V$ denote the abstract simplicial complex with vertex set $V_0$, and where all finite subsets are simplices.  Note that $V$ has a canonical $G$-action induced by the permutation action of $G$ on the coset space $V_0$.

Let $V^{(b)}$ be the barycentric subdivision (see Remark \ref{gsc rem}, part \eqref{bary rem}) of $V$; we use the barycentric subdivision to ensure that the resulting complex (is type-preserving and) admits a canonical $G$-orientation as in Remark \ref{gsc rem} part \eqref{bary rem}.
\end{definition}

Recall that an $n$-simplex $\sigma$ of $V^{(b)}$ is the same thing as a nested chain
$$
\sigma=(\sigma_0\subsetneq \sigma_1\subsetneq \cdots \subsetneq \sigma_n)
$$
of proper inclusions, where each $\sigma_i$ is a simplex of $V$ (i.e.\ a finite subset of $V_0$).  The next lemma should be compared to \cite[page 67]{Schneider:1998aa}.

\begin{lemma}\label{g mod lem}
Let $k$ be a Pr\"{u}fer domain (compare Remark \ref{k rem}).  Let $Z$ be a proper, $G$-oriented, type preserving $G$-simplicial complex.  Let $X$ be a locally compact Hausdorff $G$-space with finite $c$-$k$-cohomological dimension.  Assume moreover that $k$ contains $n^{-1}$ whenever $n$ is the order of a torsion element of $G$ that fixes a point $x\in X$.  Let 
\begin{equation}\label{v x dc}
\xymatrix{  & 0\ar[d] & 0\ar[d] & \ar[d] 0 &  \cdots  \\ 
0 & D_{-1,0} \ar[d]  \ar[l] & D_{0,0} \ar[d] \ar[l] & D_{1,0}  \ar[d] \ar[l] & \ar[l] \cdots   \\
 & \vdots \ar[d] & \vdots \ar[d] & \vdots \ar[d] & \cdots  \\
0 & D_{-1,-m}  \ar[l] \ar[d] & D_{0,-m} \ar[l] \ar[d] & D_{1,-m} \ar[d]  \ar[l] & \ar[l] \cdots  \\
 & 0  & 0 & 0  & \cdots}
\end{equation}
denote the double cochain complex from Definition \ref{dc2} above. Then for $p\geq 0$, each module $D_{p,q}$ is flat and each row is exact.
\end{lemma}

\begin{proof}
We first show flatness of the $G$-$k$-modules $D_{p,q}$ for $p\geq 0$.  Note first that for each $\sigma\in V_p^{(b)}$ that $\Gamma_c(X_\sigma,I^{-q})$ is flat as a $k$-module by Lemma \ref{rgcx}, part \eqref{rgcx 2}.  It is thus also flat as a $G_\sigma$-$k$ module by Lemma \ref{co inv}.  Each of the modules $\bigoplus_{\sigma\in V_i^{(b)}}\Gamma_c(X_\sigma,I^{-q})$ is a direct sum of $G$-$k$-modules that are induced from flat modules of this form (one for each $G$-orbit in $V_i^{(b)}$), so it is a direct sum of flat modules by Lemma \ref{g-mod proj}, so flat.

Fix $q$.  It remains to show that the $q^\text{th}$ row in line \eqref{v x dc} is exact.  For this it suffices to show that it is exact when considered in the category of $k$-modules (i.e.\ after forgetting the $G$-action).  We will in fact show that it is chain contractible in the category of $k$-modules.

Recall that $V_0=\bigsqcup_{H\in \text{Fin}(G)}G/H$ is the vertex set of $V$.  We have that $(X_\alpha)_{\alpha\in V_0}$ is an open (and closed) cover of $\widehat{X}$ as every finite subgroup of $G$ appears as the stabilizer of some $\alpha\in V_0$.  Let $\{U_\alpha\}_{\alpha\in V_0}$ be a partition of $\widehat{X}$ into clopen sets (with some $U_\alpha$ possibly empty) such that $U_\alpha\subseteq X_\alpha$ for all $\alpha$; to see that such exists, we can for example first find a partition $\{P_\alpha\}_{\alpha\in V_0}$ of $G_{tor}$ such that that each $P_\alpha$ is contained in $G_\alpha$, and then set $U_\alpha:=\phi_X^{-1}(P_\alpha)$.  For each $\alpha$, each $q$ and each clopen set $V$ of $\widehat{X}$, let $s_\alpha:\Gamma_c(V,I^q)\to \Gamma_c(V,I^q)$ be the map that restricts a section to $V\cap U_\alpha$.  

Let now $\sigma=(\sigma_0\subsetneq \sigma_1\subsetneq \cdots \subsetneq \sigma_n)\in V_n^{(b)}$ for some $n$.  Let $a\in \Gamma_c(X_\sigma,I^q)$.  Define $s(a)\in \bigoplus _{\eta\in V_{n+1}^{(b)}}\Gamma_c(X_\eta,I^q)$ by stipulating that the component of $s(a)$ in $k[X_\eta]$ is
$$
\left\{\begin{array}{ll} (-1)^i s_\alpha a, & \eta=\sigma_0\subsetneq \cdots \subsetneq \sigma_{i-1} \subsetneq \sigma_{i-1}\sqcup\{\alpha\}\subsetneq \sigma_i\subsetneq \cdots \subsetneq \sigma_n \text{ for some } \\ &  i\in \{0,...,n+1\} \text{ and } \alpha\in V_0 \\
0, & \text{otherwise}\end{array}\right.
$$
This does define an element of $\bigoplus _{\eta\in V_{n+1}^{(b)}}\Gamma_c(X_\eta,I^q)$: indeed, if $\eta=\sigma_0\subsetneq \cdots \subsetneq \sigma_{i-1} \subsetneq \sigma_{i-1}\sqcup\{\alpha\}\subsetneq \sigma_i\subsetneq \cdots \subsetneq \sigma_n$ for some $i\in \{0,...,n+1\}$ and $\alpha\in V$ as above, then $X_\alpha\cap X_\sigma=X_\eta$, and so in particular $s(a)$ is supported in $\Gamma_c(X_\eta,I^q)$; moreover, $s(a)$ can only be non-zero for finitely many components $\Gamma_c(X_\eta,I^q)$ as only finitely many images $s_\alpha a$ can be non-zero by compact support of $a$.  Now define 
$$
s:\bigoplus _{\sigma\in V_{n}^{(b)}}\Gamma_c(X_\sigma,I^q)\to \bigoplus _{\eta\in V_{n+1}^{(b)}}\Gamma_c(X_\eta,I^q)
$$
by extending by linearity; note that $s$ is a map of $k$ modules, but in general not of $G$-$k$ modules.

To complete the proof, it suffices to show that $s$ is a chain contraction of the $q^{\text{th}}$ row of the double complex in line \eqref{v x dc} (considered in the category of $k$-modules), i.e.\ that $s\partial^{(h)}+\partial^{(h)} s=\text{id}$.  Let $a\in \Gamma_c(X_\sigma,I^q)$ for some $\sigma=(\sigma_0\subsetneq \sigma_1\subsetneq \cdots \subsetneq \sigma_n)\in V_n^{(b)}$ and some $n$.  For $j\in \{0,...,n\}$, write 
$$
T_j:=\{\alpha\in V_0\mid \sigma_j=\sigma_{j-1}\sqcup \{\alpha\}\}
$$ 
(if $j=0$, the condition ``$\sigma_j=\sigma_{j-1}\sqcup \{\alpha\}$'' should be read as ``$\sigma_0=\{\alpha\}$''), so each $T_j$ is either empty or a singleton.  We compute then that the component of $s\partial^{(h)}(a)$ in $\Gamma_c(X_\eta,I^q)$ is 
{\tiny $$
\left\{\begin{array}{ll} (-1)^i(-1)^js_\alpha a, & \eta=\sigma_0\subsetneq \cdots \subsetneq \sigma_{i-1} \subsetneq \sigma_{i-1}\sqcup\{\alpha\}\subsetneq \sigma_i\subsetneq  \cdots \subsetneq \sigma_{j-1}\subsetneq \sigma_{j+1}\subsetneq \cdots \subsetneq \sigma_n \\
\sum_{j=0}^n(-1)^j(-1)^j\sum_{\alpha\in T_j} s_{\alpha}a, & \eta=\sigma \\
(-1)^j(-1)^js_{\alpha} a, & \eta=\sigma_0\subsetneq \cdots \subsetneq \sigma_{j-1} \subsetneq \sigma_{j-1}\sqcup\{\alpha\}\subsetneq \sigma_{j+1} \subsetneq \sigma_n\neq \sigma \\
(-1)^j(-1)^{i-1} s_{\alpha} a, & \eta=\sigma_0\subsetneq \cdots \subsetneq \sigma_{j-1}\subsetneq \sigma_{j+1}\subsetneq \cdots \subsetneq \sigma_{i-1} \subsetneq \sigma_{i-1}\sqcup\{\alpha\}\subsetneq \sigma_i\subsetneq \cdots \subsetneq \sigma_n
 \\
0, & \text{otherwise}\end{array}\right..
$$}
On the other hand, for $i\in \{0,...,n+1\}$ we define 
$$
S_i:=\{\alpha\in V_0\mid \sigma_{i-1}\subsetneq \sigma_{i-1}\sqcup\{\alpha\}\subsetneq \sigma_i\}
$$
(if $i=0$ (respectively $i=n+1$), the condition ``$\sigma_{i-1}\subsetneq \sigma_{i-1}\sqcup\{\alpha\}\subsetneq \sigma_i$'' should be read as ``$\{\alpha\}\subsetneq \sigma_0$'' (respectively, ``$\sigma_n\subsetneq \sigma_n\sqcup\{\alpha\}$'', i.e.\ $\alpha\not\in \sigma_n$)) and then the component of $\partial^{(h)} s(a)$ in $\Gamma_c(X_\eta,I^q)$ is 
{\tiny $$
\left\{\begin{array}{ll} (-1)^i(-1)^{j+1}s_\alpha a, & \eta=\sigma_0\subsetneq \cdots \subsetneq \sigma_{i-1} \subsetneq \sigma_{i-1}\sqcup\{\alpha\}\subsetneq \sigma_i\subsetneq  \cdots \subsetneq \sigma_{j-1}\subsetneq \sigma_{j+1}\subsetneq \cdots \subsetneq \sigma_n \\
\sum_{i=0}^{n+1}(-1)^i(-1)^i \sum_{\alpha\in S_i} s_{\alpha}a & \eta=\sigma \\
(-1)^i(-1)^{i+1}s_{\alpha} a, & \eta=\sigma_0\subsetneq \cdots \subsetneq \sigma_{i-1} \subsetneq \sigma_{i-1}\sqcup\{\alpha\}\subsetneq \sigma_{i+1} \subsetneq \sigma_n\neq \sigma \\
(-1)^j(-1)^{j} s_{\alpha} a, & \eta=\sigma_0\subsetneq \cdots \subsetneq \sigma_{j-1}\subsetneq \sigma_{j+1}\subsetneq \cdots \subsetneq \sigma_{i-1} \subsetneq \sigma_{i-1}\sqcup\{\alpha\}\subsetneq \sigma_i\subsetneq \cdots \subsetneq \sigma_n
 \\
0, & \text{otherwise}\end{array}\right.
$$}
These computations show that $s\partial^{(h)}(a)+\partial^{(h)} s(a)=a$ for such $a$; and as such $a$ generate $\bigoplus _{\sigma\in V_{n}^{(b)}}\Gamma_c(X_\sigma,I^q)$ as a $k$-module, we are done.
\end{proof}

\begin{proof}[Proof of Proposition \ref{bcr and gh}]
With notation as in Lemma \ref{g mod lem}, we have $|V^{(b)}|=|V|$.  Moreover, $|V|$ is a well-known simplicial model for $\underline{E}G$: indeed, the infinite join construction for $\underline{E}G$ from \cite[Appendix 1]{Baum:1994pr} is exactly $|V|$ in the case that $G$ is discrete.  Hence $H^*_{G,k,!}(\underline{E}G,X)=H^*_{G,k,!}(|V^{(b)}|,X)$.  The group $H^*_{G,k,!}(|V^{(b)}|,X)$ can be computed as the homology of the complex from Corollary \ref{! cor}.  

On the other hand, $H_*(G,\Gamma_c(I^{-\bullet}))$ can be computed by replacing $\Gamma_c(\widehat{X},I^{-\bullet})$ with a quasi-isomorphic bounded below chain complex of flat modules, taking coinvariants, and then taking homology: indeed, in the language of \cite[Definition 10.5.4]{Weibel:1995ty}, flat $G$-$k$ modules are acyclic for the coinvariants functor, and we may thus apply \cite[Generalized Existence Theorem 10.5.9]{Weibel:1995ty}\footnote{This is stated for total right derived functors, but there is an analogous version for total left derived functors.} to compute the total left derived functor of the coinvariants functor.  To replace $\Gamma_c(\widehat{X},I^{-\bullet})$ with a quasi-isomorphic bounded below chain complex of flat modules, note that Lemma \ref{g mod lem} implies that the double complex
\begin{equation}\label{v x dc2}
\xymatrix{  & 0\ar[d] & 0\ar[d] & \ar[d] 0 &  \cdots  \\ 
0 & D_{0,0} \ar[d]  \ar[l] & D_{1,0} \ar[d] \ar[l] & D_{2,0}  \ar[d] \ar[l] & \ar[l] \cdots   \\
 & \vdots \ar[d] & \vdots \ar[d] & \vdots \ar[d] & \cdots  \\
0 & D_{0,-m}  \ar[l] \ar[d] & D_{1,-m} \ar[l] \ar[d] & D_{2,-m} \ar[d]  \ar[l] & \ar[l] \cdots  \\
 & 0  & 0 & 0  & \cdots}
\end{equation}
consists of flat modules; moreover, as the augmentation of this double complex in line \eqref{v x dc} has exact rows (the proof of) the so-called acyclic assembly lemma \cite[2.7.3]{Weibel:1995ty} shows that the (direct sum) total complex of the double complex in line \eqref{v x dc2} is quasi-isomorphic to the first column $D_{-1,-m}\leftarrow \cdots \leftarrow D_{-1,0}$ of the double complex in line \eqref{v x dc}; this in turn is exactly the complex $\Gamma_c(\widehat{X},I^{-\bullet})$ we are interested in.  As direct sums of flat modules are flat, the (direct sum) total complex is therefore a bounded below chain complex of flat modules that is quasi-isomorphic to $\Gamma_c(\widehat{X},I^{-\bullet})$  as required.  

To summarize the above paragraph, we may compute $H_*(G,\Gamma_c(I^{-\bullet}))$ by starting with the (direct sum) total complex of the double complex in line \eqref{v x dc2}, taking coinvariants, and then taking homology.  Up to replacing the index ``$n$'' by ``$-n$'', this gives exactly the same complex as the total complex from Corollary \ref{! cor}, however, so we are done.  
\end{proof}

\end{document}